\documentclass[a4paper, 12pt]{amsart}
\usepackage{hyperref,upref,amssymb,tikz}

\numberwithin{equation}{section}

\newtheorem{thm}{Theorem}[section]
\newtheorem{lemma}[thm]{Lemma}
\newtheorem{prop}[thm]{Proposition}

\theoremstyle{definition}
\newtheorem{defn}[thm]{Definition}
\newtheorem{ntn}[thm]{Notation}
\theoremstyle{remark}
\newtheorem{eg}[thm]{Example}
\newtheorem{rmk}[thm]{Remark}

\newcommand{\supp}{\operatorname{supp}}

\newcommand{\NN}{\mathbb{N}}

\newcommand{\RR}{\mathbb{R}}
\newcommand{\TT}{\mathbb{T}}
\newcommand{\id}{\operatorname{id}}

\newcommand{\Oo}{\mathcal{O}}

\newcommand{\dom}{\mathrm{dom}}
\newcommand{\ran}{\mathrm{ran}}

\title{Twisted Topological Graph Algebras Are Twisted Groupoid $C^*$-algebras}

\date{23 Jan 2017}

\author{Alex Kumjian}
\email{alex@unr.edu}
\address{Department of Mathematics, University of Nevada (084), Reno, NV, 89557, United States}

\author{Hui Li}
\email{lihui8605@hotmail.com}
\address{Research Center for Operator Algebras and Shanghai Key Laboratory of Pure Mathematics and Mathematical Practice, Department of Mathematics, East China Normal University, 3663 Zhongshan North Road, Putuo District, Shanghai, 200062, CHINA}

\thanks{The second author is the corresponding author.}
\thanks{The first author was partially supported by Simons Collaboration Grant \#\,353626. The second author was supported by Research Center for Operator Algebras of East China Normal University and was supported by Science and Technology Commission of Shanghai Municipality (STCSM), grant No. 13dz2260400.}

\begin{document}

\begin{abstract}
In \cite{HuiLiTwisted1} the second author showed how Katsura's construction of the $C^*$-algebra of a topological graph $E$ may be twisted by a Hermitian line bundle $L$ over the edge space $E^1$.   
The correspondence defining the algebra is obtained as the completion of the compactly supported continuous sections of $L$.  We prove  that the resulting $C^*$-algebra is isomorphic to a twisted groupoid $C^*$-algebra
where the underlying groupoid is the Renault-Deaconu groupoid of the topological graph with Yeend's boundary path space as its unit space.

\end{abstract}

\maketitle

\section{Introduction}

Graph algebras have been the object of much research in operator algebras over the last twenty years or so.
Various generalizations have also been introduced and studied by numerous authors.  These include  higher-rank graph algebras introduced by Pask and the first author (see \cite{KumjianPask:NYJM00}); topological graph algebras due to Katsura (see \cite{Katsura:TAMS04}); $C^*$-algebras arising from topological quivers due to Muhly and Tomforde (see \cite{MuhlyTomforde:IJM05}); and topological higher-rank graph algebras due to Yeend (see \cite{RSWY, Yeend:JOT07, Yeend:CM06}).
These generalizations have significantly broadened the class of $C^*$-algebras brought into focus.
Twisted versions of these $C^*$-algebras have also been proposed and studied recently.
Twisted higher-rank graph algebras were introduced in \cite{MR2948223, MR3335414} where the twisting is determined by a $\mathbb{T}$-valued $2$-cocycle.
Deaconu et al.\ studied the cohomology of a groupoid  determined by a singly generated dynamical system   and the associated twisted groupoid $C^*$-algebras (see \cite{DeaconuKumjianEtAl:JOT01}).
Twisted topological graph algebras which generalize both  Katsura's topological graph algebras and the twisted groupoid $C^*$-algebras investigated in \cite{DeaconuKumjianEtAl:JOT01} were introduced by the second author in \cite{HuiLiTwisted1}.

Katsura's topological graphs may be regarded as an abstract dynamical representation of a Pimsner module
(see \cite{Pimsner:FIC97}).
The class of topological graph algebras have potential application to the classification of $C^*$-algebras because many properties of topological graph algebras may be inferred from properties of the underlying graphs.
Moreover, topological graphs provide models for many classifiable $C^*$-algebras. Indeed, topological graph algebras include all graph algebras, all crossed products of the form $C_0(T) \rtimes_\alpha \mathbb{Z}$ (see \cite{Katsura:TAMS04}), all AF-algebras, all A$\mathbb{T}$-algebras, many AH-algebras, Renault-Deaconu groupoid $C^*$-algebras arising from a singly generated dynamical system, etc. (see \cite{Katsura:IJM06}).  By a celebrated result  all simple, separable, nuclear, purely infinite $C^*$-algebras satisfying the UCT are topological graph algebras (see \cite{Katsura:JFA08}).

Twisted topological graph algebras also have applications to the field of noncommutative geometry.
Recently, Kang et al.\  proved that all quantum Heisenberg manifolds may be realized as twisted topological graph algebras (see \cite{MR3303906}).

A \emph{partial local homeomorphism} on a locally compact Hausdorff space $T$ is defined to be a local homeomorphism $\sigma: \dom(\sigma) \to \ran(\sigma)$ where $\dom(\sigma), \ran(\sigma)$ are open subsets of $T$.
The pair $(T,\sigma)$ is called a \emph{singly generated dynamical system}.
Given a singly generated dynamical system $(T,\sigma)$, one may define the \emph{Renault-Deaconu groupoid} $\Gamma(T,\sigma)$ which is  both \'{e}tale and amenable (see \cite{Deaconu:TAMS95, Renault:00}).

Recall that graph algebras associated to row-finite directed graphs with no sources were realized  as Renault-Deaconu groupoid $C^*$-algebras (see  \cite{KumjianPaskEtAl:PJM98, KumjianPaskEtAl:JFA97}).
Note that  the $C^*$-algebra of an arbitrary graph is not defined as a groupoid $C^*$-algebra but as the universal  $C^*$-algebra of a family of generators indexed by the vertices and edges of a graph subject to Cuntz-Krieger type relations (see \cite{BatesHongEtAl:IJM02, DrinenTomforde:RMJM05, FowlerLacaEtAl:PAMS00, RaeburnSzyma'nski:TAMS04}, etc).
Katsura's definition of topological graph algebras is based on a modified model of Cuntz-Pimsner algebras (see \cite{Katsura:JFA04, Pimsner:FIC97}). He showed in \cite{MR2563503} that when vertex and edge spaces of a topological graph are both compact and the range map is surjective, then the topological graph algebra is isomorphic to a Renault-Deaconu groupoid $C^*$-algebra, and conjectured  that this is true more generally.
Yeend proved that every topological graph algebra is indeed a groupoid $C^*$-algebra (see \cite{Yeend:CM06}).

Our main result in the present work (see Theorem~\ref{twisted top graph alg iso to twisted gpoid C*-alg}) is that every twisted topological graph algebra is isomorphic to a twisted groupoid $C^*$-algebra
(see Definition~\ref{define the twisted groupoid C^*-alg}) and that the underlying groupoid is indeed the canonical Renault-Deaconu groupoid associated to a shift map  with Yeend's boundary path space as its unit space (this was implicit in Yeend's work but requires some work to tease out).
This result implies that every topological graph algebra is isomorphic to a Renault-Deaconu groupoid $C^*$-algebra,
thereby confirming Katsura's conjecture.

We start this paper with three equivalent definitions of twisted topological graph algebras in Section~2.
Then in Section~3 we recall from \cite{Katsura:TAMS04, HuiLiTwisted1} some basic terminology of topological graphs and some fundamental results about twisted topological graph algebras.
In Section~4, we introduce a notion of boundary path which is based on Webster's definition in the case of a directed graph (see \cite{MR3119197}), and prove that our definition coincides with Yeend's definition of boundary path of a topological higher-rank graphs when restricted to topological $1$-graph (see \cite{Yeend:CM06}). In Section~5 we use Katsura's factor map technique from \cite{Katsura:IJM06} to construct homomorphisms between twisted topological graph algebras. In Section~6, we obtain the relationship between principal circle bundles over the domain of a partial local homeomorphism and topological twists over the Renault-Deaconu groupoid arising from the given partial local homeomorphism. We conclude in Section~7 by proving our main result, Theorem~\ref{twisted top graph alg iso to twisted gpoid C*-alg}, which says that every twisted topological graph algebra is isomorphic to a twisted groupoid $C^*$-algebra where the underlying groupoid is the Renault-Deaconu groupoid of the topological graph discussed above.

\section{Three Equivalent Definitions}

In this section, we recall the notion of twisted topological graph algebras introduced by Li in \cite{HuiLiTwisted1} and also give other equivalent descriptions of this type of $C^*$-algebras.

\begin{defn}[{\cite{Katsura:TAMS04}}]\label{define the top graph}
A quadruple $E=(E^0,E^1,r,s)$ is called a \emph{topological graph} if $E^0, E^1$ are locally compact Hausdorff spaces, $r:E^1 \to E^0$ is a continuous map, and $s:E^1\to E^0$ is a local homeomorphism.
\end{defn}

Now we introduce the construction of twisted topological graph algebras from different point of views. Our construction involves $C^*$-correspondences and Cuntz-Pimsner algebras (see \cite{Katsura:JFA04, Lance:Hilbert$C*$-modules95, Pimsner:FIC97, Raeburn:Graphalgebras05, RaeburnWilliams:Moritaequivalenceand98}, etc).

Let $E$ be a topological graph, let $\mathbf{N}=\{N_\alpha\}_{\alpha \in \Lambda}$ be an open cover of $E^1$, and let $\mathbf{S}=\{s_{\alpha\beta} \in C(\overline{N_{\alpha\beta}},\mathbb{T})\}_{\alpha, \beta \in \Lambda}$ be a \emph{1-cocycle}, which is a collection of circle-valued continuous functions such that $s_{\alpha\beta}s_{\beta\gamma}=s_{\alpha\gamma}$ on $\overline{N_{\alpha\beta\gamma}}$. Suppose that $x, y \in \prod_{\alpha \in \Lambda} C(\overline{N_\alpha})$ satisfy $x_\alpha=s_{\alpha\beta}x_\beta$ and $y_\alpha=s_{\alpha\beta}y_\beta \ \mathrm{on}\ \overline{N_{\alpha\beta}}$. Define $[x \vert y] \in C(E^1)$ by
\begin{align*}
[x \vert y](e)=\overline{x_\alpha(e)}y_\alpha(e),\ \mathrm{if} \ e \in N_\alpha.
\end{align*}
By \cite[Definition~3.2]{HuiLiTwisted1}, define
\begin{align*}
C_c({E,\mathbf{N},\mathbf{S}}):=\Big\{x \in \prod_{\alpha \in \Lambda} C(\overline{N_\alpha}): x_\alpha=s_{\alpha\beta}x_\beta \ \mathrm{on}\ \overline{N_{\alpha\beta}}, [x \vert x] \in C_c(E^1) \Big\}.
\end{align*}
For $x, y \in C_c({E,\mathbf{N},\mathbf{S}}), \alpha \in \Lambda, f \in C_0(E^0)$, and for $v \in E^0$, define
\begin{align*}
(x\cdot f)_\alpha &:=x_\alpha (f \circ s \vert_{\overline{N_\alpha}}); \\
(f \cdot x)_\alpha &:=(f \circ r \vert_{\overline{N_\alpha}})x_\alpha; \quad \text{ and } \\
 \langle x,y \rangle_{C_0(E^0)}(v) &:=\sum_{s(e)=v}[x\vert y](e).
\end{align*}
By \cite[Theorem~3.3]{HuiLiTwisted1}, $C_c(E,\mathbf{N},\mathbf{S})$ is a right inner product $C_0(E^0)$-module with an adjointable left $C_0(E^0)$-action, and its completion $X(E,\mathbf{N},\mathbf{S})$ under the $\Vert\cdot\Vert_{C_0(E^0)}$-norm is a $C^*$-correspondence over $C_0(E^0)$. We denote $\mathcal{O}(E,\mathbf{N},\mathbf{S})$ the Cuntz-Pimsner algebra of $X(E,\mathbf{N},\mathbf{S})$ (see Notation~\ref{define Cunz-Pimsner alg}).

Let $E$ be a topological graph and let $p:L \to E^1$ be a Hermitian line bundle. Then each fibre has a one-dimensional Hilbert space structure conjugate linear in the first variable, and the map $\{(l_1,l_2) \in L \times L: p(l_1)=p(l_2) \} \to \mathbb{C}$ by sending $(l_1,l_2)$ to $\langle l_1,l_2 \rangle_{p(l_1)}$ is continuous. For two continuous sections $x,y$ of $L$, there is a continuous function $[x \vert y]:E^1 \to \mathbb{C}$ by $[x \vert y](e):=\langle x(e),y(e) \rangle_e$. Define $C_c(E,L)$ to be the set of all continuous sections $x$ satisfying that $[x\vert x] \in C_c(E^1)$. Then $C_c(E,L)$ has a natural vector space structure. For $x,y \in C_c(E,L), f \in C_0(E^0), e \in E^1$, and for $v \in E^0$, define
\[
(x \cdot f)(e):=x(e)f \circ s(e);\ (f \cdot x)(e):=f \circ r(e)x(e); \text{ and }
\]
\[
\langle x,y \rangle_{C_0(E^0)}(v):=\sum_{s(e)=v}[x \vert y](e).
\]
It is straightforward to check that $C_c(E,L)$ is a right inner product $C_0(E^0)$-module with an adjointable left $C_0(E^0)$-action, its completion $X(E,L)$ under the $\Vert\cdot\Vert_{C_0(E^0)}$-norm is a $C^*$-correspondence over $C_0(E^0)$. Denote by $\mathcal{O}(E,L)$ the Cuntz-Pimsner algebra of $X(E,L)$ (see Notation~\ref{define Cunz-Pimsner alg}).

Let $E$ be a topological graph and let $p:\mathbf{B} \to E^1$ be a principal circle bundle. By \cite[Proposition~4.65]{RaeburnWilliams:Moritaequivalenceand98}, there exists a collection of continuous local sections $\{s_\alpha:N_\alpha \to \mathbf{B}\}_{\alpha \in \Lambda}$ at each point of $E^1$. Denote the set of all equivariant functions in $C_c(\mathbf{B})$ by $C_c^{e}(\mathbf{B})$. For $x,y \in C_c^e(\mathbf{B})$, and for $e \in E^1$, define $[x \vert y](e):=\overline{x \circ s_\alpha(e)}y \circ s_\alpha(e)$ if $e \in N_\alpha$. Then $[x \vert y] \in C_c(E^1)$. By \cite[Page~258]{DeaconuKumjianEtAl:JOT01}, for $x,y \in C_c^e(\mathbf{B}), f \in C_0(E^0), b \in \mathbf{B}$, and for $v \in E^0$, define
\[
(x \cdot f)(b):=x(b)f(s(p(b)));\ (f \cdot x)(b):=f(r(p(b)))x(b); \text{ and }
\]
\[
\langle x,y \rangle_{C_0(E^0)}(v):=\sum_{s(e)=v}[x \vert y](e).
\]
Then $C_c^e(\mathbf{B})$ is a right inner product $C_0(E^0)$-module with an adjointable left $C_0(E^0)$ action, its completion $X(E,\mathbf{B})$ under the $\Vert\cdot\Vert_{C_0(E^0)}$-norm is a $C^*$-correspondence over $C_0(E^0)$. Denote by $\mathcal{O}(E,\mathbf{B})$ the Cuntz-Pimsner algebra of $X(E,\mathbf{B})$ (see Notation~\ref{define Cunz-Pimsner alg}).

\begin{prop}\label{X(E,N,S) is iso to X(E,L)}
Let $E$ be a topological graph, let $\mathbf{N}=\{N_\alpha\}_{\alpha \in \Lambda}$ be an open cover of $E^1$, and let $\mathbf{S}=\{s_{\alpha\beta} \in C(\overline{N_{\alpha\beta}},\mathbb{T})\}_{\alpha, \beta \in \Lambda}$ be a collection of circle-valued continuous functions such that for $\alpha, \beta, \gamma \in \Lambda, s_{\alpha\beta}s_{\beta\gamma}=s_{\alpha\gamma}$ on $\overline{N_{\alpha\beta\gamma}}$. Define a Hermitian line bundle over $E^1$ by (with the projection map $p$)
\[
L:=\amalg_{\alpha \in \Lambda}(N_\alpha \times \mathbb{C})/(e,z,\alpha) \sim (e,s_{\beta\alpha}(e)z,\beta).
\]
Then $X(E,\mathbf{N},\mathbf{S})$ and $X(E,L)$ are isomorphic as $C^*$-correspondences over $C_0(E^0)$.
\end{prop}
\begin{proof}
We define a map $\Phi:C_c(E,\mathbf{N},\mathbf{S}) \to X(E,L)$ by $\Phi(x)(e):=(e,x_\alpha(e)$, $\alpha)$, for all $x \in C_c(E,\mathbf{N},\mathbf{S})$ and for all $e \in N_\alpha$. It is straightforward to check that $\Phi$ preserves $C_0(E^0)$-valued inner products and module actions. So there exists a unique extension of $\Phi$ to $X(E,\mathbf{N},\mathbf{S})$ which preserves $C_0(E^0)$-valued inner products and module actions. We still denote the extension by $\Phi$. Fix $\alpha_0 \in \Lambda$, and fix $x \in C_c(E,L)$ such that $\supp([x \vert x]) \subset N_{\alpha_0}$. By a partition of unity argument, it is sufficient to show that $x$ is in the image of $\Phi$. Let $f$ be the composition of $x \vert_{N_{\alpha_0}}$ and the projection from $N_{\alpha_0}\times \mathbb{C} \times \{\alpha_0\}$ onto $\mathbb{C}$. Then $f \in C_0(N_{\alpha_0})$. As in \cite[Page~5]{HuiLiTwisted1}, there exists $(x_\alpha) \in C_c(E,\mathbf{N},\mathbf{S})$, such that
\begin{align*}
x_\alpha(e) := \begin{cases}
    s_{\alpha\alpha_0}(e)f(e) &\text{ if $e \in \overline{N_{\alpha\alpha_0}}$} \\
    0 &\text{ if $e \in \overline{N_\alpha} \setminus \overline{N_{\alpha\alpha_0}}$ }.
\end{cases}
\end{align*}
It is straightforward to check that $\Phi(x_\alpha)=x$ and we are done.
\end{proof}

\begin{prop}\label{X(E,N,S) is iso to L_B}
Let $E$ be a topological graph, let $\mathbf{N}=\{N_\alpha\}_{\alpha \in \Lambda}$ be an open cover of $E^1$, and let $\mathbf{S}=\{s_{\alpha\beta}\}_{\alpha,\beta \in \Lambda}$ be a $1$-cocycle relative to $\mathbf{N}$. Let $\mathbf{B}:=\amalg_{\alpha \in \Lambda}(N_\alpha \times \mathbb{T}) / (e,z,\alpha) \sim (e,z s_{\alpha\beta}(e),\beta)$ be the corresponding principal circle bundle. Then $X(E$, $\mathbf{N},\mathbf{S})$ and $X(E,\mathbf{B})$ are isomorphic as $C^*$-correspondences over $C_0(E^0)$.
\end{prop}
\begin{proof}
We denote the projection by $p:\mathbf{B} \to E^1$. We define a map $\Phi:C_c(E,\mathbf{N},\mathbf{S}) \to X(E,\mathbf{B})$ by $\Phi(x)(e,z,\alpha):=zx_\alpha(e)$, for all $x \in C_c(E,\mathbf{N},\mathbf{S})$ and for all $(e,z,\alpha) \in \mathbf{B}$. It is straightforward to check that $\Phi$ preserves $C_0(E^0)$-valued inner products and module actions. So there exists a unique extension of $\Phi$ to $X(E,\mathbf{N},\mathbf{S})$ which preserves $C_0(E^0)$-valued inner products and module actions.
Let $\Phi$ also denote the extension.
Fix $\alpha_0 \in \Lambda$ and $x \in C_c^e(\mathbf{B})$ with $p(\supp(x)) \subset N_{\alpha_0}$. By a partition of unity argument, it is sufficient to show that $x$ is in the image of $\Phi$. Let $f$ be the composition of the continuous local section $s_{\alpha_0}:N_{\alpha_0} \to \mathbf{B}$ satisfying that $s_{\alpha_0}(e):=(e,1,\alpha_0)$ and $x$. Then $f \in C_c(N_{\alpha_0})$. By the construction in \cite[Page~5]{HuiLiTwisted1}, there exists $(y_\alpha) \in C_c(E,\mathbf{N},\mathbf{S})$, such that
\begin{align*}
y_\alpha(e) := \begin{cases}
    s_{\alpha\alpha_0}(e)f(e) &\text{ if $e \in \overline{N_{\alpha\alpha_0}}$} \\
    0 &\text{ if $e \in \overline{N_\alpha} \setminus \overline{N_{\alpha\alpha_0}}$ }.
\end{cases}
\end{align*}
It is straightforward to check that $\Phi(y_\alpha)=x$ and we are done.
\end{proof}

\begin{rmk}
In \cite{HuiLiTwisted1}, $X(E,\mathbf{N},\mathbf{S})$ is called the twisted graph correspondence and the Cuntz-Pimsner algebra $\mathcal{O}(E,\mathbf{N},\mathbf{S})$ is called the twisted topological graph algebra. By Propositions~\ref{X(E,N,S) is iso to X(E,L)}, \ref{X(E,N,S) is iso to L_B}, any form of $X(E,\mathbf{N},\mathbf{S}), X(E,L), X(E,\mathbf{B})$ can be used as the definition of the twisted graph correspondence, and any form of $\mathcal{O}(E,\mathbf{N},\mathbf{S}), \mathcal{O}(E,L), \mathcal{O}(E,\mathbf{B})$ can be used as the definition of the twisted topological graph algebra.

In this paper, we call $X(E,\mathbf{B})$ the \emph{twisted graph correspondence} associated to $E$ and $\mathbf{B}$, and we call $\mathcal{O}(E,\mathbf{B})$ the \emph{twisted topological graph algebra}.
\end{rmk}

\section{Twisted Topological Graph Algebras}

In this section, we recap the terminology of topological graphs from \cite{Katsura:TAMS04} and recall some fundamental results about twisted topological graph algebras from \cite{HuiLiTwisted1}.

Let $E$ be a topological graph. A subset $U$ of $E^1$ is called an \emph{$s$-section} if $s\vert_{U}:U \to s(U)$ is a homeomorphism with respect to the subspace topologies. Define $E_{\mathrm{fin}}^0$ to be the subset of all $v \in E^0$ which has an open neighborhood $N$ such that $r^{-1}(\overline{N})$ is compact; define $E_{\mathrm{sce}}^0:=E^0 \setminus \overline{r(E^1)}$; define $E_{\mathrm{rg}}^0:=E_{\mathrm{fin}}^0 \setminus \overline{E_{\mathrm{sce}}^0}$; and define $E_{\mathrm{sg}}^0:=E^0 \setminus E_{\mathrm{rg}}^0$.

The sets $E_{\mathrm{fin}}^0, E_{\mathrm{sce}}^0, E_{\mathrm{rg}}^0$ are all open, and the set $E_{\mathrm{sg}}^0$ is closed.

Denote by $r^0:=\id, s^0:=\id$, and define a topological graph $E_0:=(E^0,E^0$, $r^0,s^0)$. Denote by $r^1:=r, s^1:=s, E_1:=(E^0,E^1,r^1,s^1)=E$.

For $n \geq 2$, define
\[
E^n:=\Big\{\mu =(\mu_1,\dots,\mu_n)\in \prod_{i=1}^{n}E^1: s(\mu_i)=r(\mu_{i+1}), i=1,\dots,n-1 \Big\}
\]
endowed with the subspace topology of the product space $\prod_{i=1}^{n}E^1$. Define $r^n:E^n \to E^0$ by $r^n(\mu):=r(\mu_1)$, which is a continuous map. Define $s^n:E^n \to E^0$ by $s^n(\mu):=s(\mu_n)$, which is a local homeomorphism. Define a topological graph $E_n:=(E^0,E^n,r^n,s^n)$.

Define the \emph{finite-path space} $E^*:=\coprod_{n=0}^{\infty} E^n$ with the disjoint union topology. Define a continuous map $r:E^* \to E^0$ by $r(\mu):=r^n(\mu)$ if $\mu \in E^n$, define a local homeomorphism $s:E^n \to E^0$ by $s(\mu):=s^n(\mu)$ if $\mu \in E^n$, and define a topological graph $E_*:=(E^0,E^*,r,s)$.

Define the \emph{infinite path space}
\[
E^\infty:=\Big\{\mu\in \prod_{i=1}^{\infty}E^1: s(\mu_i)=r(\mu_{i+1}), i=1,2,\dots \Big\}.
\]
Define the range map $r:E^\infty \to E^0$ by $r(\mu):=r(\mu_1)$.

Denote the length of a path $\mu \in E^* \amalg E^\infty$ by $\vert\mu\vert$.

In discussing Cuntz-Pimsner algebras associated with correspondences we follow the conventions of
\cite{Katsura:JFA04} and \cite[Chapter~8]{Raeburn:Graphalgebras05}.

\begin{ntn}\label{define Cunz-Pimsner alg} 
Let $E$ be a topological graph and let $p:\mathbf{B} \to E^1$ be a principal circle bundle. Let $\phi:C_0(E^0) \to \mathcal{L}(X(E,\mathbf{B}))$ denote the homomorphism determined by the left action. Define $J_{X(E,\mathbf{B})}:=\{f \in C_0(E^0):f\in\phi^{-1}(\mathcal{K}(X(E,\mathbf{B})))\cap(\ker\phi)^\perp\}$, which is a closed two-sided ideal of $C_0(E^0)$. A pair $(\psi, \pi)$ consisting of a linear map $\psi:X(E,\mathbf{B}) \to B$ and a homomorphism $\pi:C_0(E^0) \to B$ defines a (Toeplitz) \emph{representation} of $X(E,\mathbf{B})$ into a $C^*$-algebra $B$ if
\[
\psi(f \cdot x)=\pi(f)\psi(x) \quad\text{and} \quad \psi(x)^*\psi(y)=\pi(\langle x,y \rangle_{C_0(E^0)})
\]
for all $x,y \in X(E,\mathbf{B}),f\in C_0(E^0)$.
In this case there exists a unique homomorphism $\psi^{(1)}:\mathcal{K}(X(E,\mathbf{B})) \to B$ such that $\psi^{(1)}(\Theta_{x,y})=\psi(x)\psi(y)^*$.
We say that $(\psi,\pi)$ is  \emph{covariant} if $\pi(f)=\psi^{(1)}(\phi(f))$ for all $f \in J_{X(E,\mathbf{B})}$.
The representation $(\psi,\pi)$ is said to be \emph{universal covariant} if for any covariant  representation $(\psi',\pi')$ of $X(E,\mathbf{B})$ into a $C^*$-algebra $C$, there exists a unique homomorphism $h:B \to C$ such that $h\circ\psi=\psi',h\circ\pi=\pi'$. The $C^*$-algebra generated by the images of a universal covariant  representation of $X(E,\mathbf{B})$ is called the \emph{Cuntz-Pimsner algebra} associated to $E,\mathbf{B}$;
it is denoted by $\mathcal{O}(E,\mathbf{B})$.
\end{ntn}

\begin{prop}[{\cite[Proposition~3.10]{HuiLiTwisted1}}]
Let $E$ be a topological graph and let $p:\mathbf{B} \to E^1$ be a principal circle bundle. Fix a nonnegative $f \in C_c(E_{\mathrm{rg}}^0)$, a finite cover $\{N_i\}_{i=1}^{n}$ of $r^{-1}(\supp(f))$ by precompact open $s$-sections with local sections $\{\varphi_i:N_i \to \mathbf{B}\}_{i=1}^{n}$, and a finite collection $\{h_i\}_{i=1}^{n} \subset C_c(E^1, [0,1])$ satisfying $\supp(h_i) \subset N_i$ and $\sum_{i=1}^{n}h_i=1$ on $r^{-1}(\supp(f))$. For $i$, for $b \in p^{-1}(N_i)$, define $x_i \in C_c^e(p^{-1}(N_i))$ by $x_i(b):=b / (\varphi_i \circ p(b)) \sqrt{h_i \circ p(b) f \circ r \circ p(b)}$. Then
\[
\phi(f)=\sum_{i=1}^{n}\Theta_{x_i,x_i}.
\]
\end{prop}

Finally, we recall some operations on a principal circle bundle from \cite{DeaconuKumjianEtAl:JOT01}. Let $T, T_1, T_2$ be locally compact Hausdorff spaces and let $p:\mathbf{B} \to T,p_i:\mathbf{B}_i \to T_i, i=1,2$ be principal circle bundles. For $b, b'$ in the same fibre of $\mathbf{B}$, there exists a unique $b/b' \in \mathbb{T}$ such that $b=(b/b') \cdot b'$. There exists a conjugate principal circle bundle $\overline{\mathbf{B}}$ over $T$ together with a homeomorphism $\mathbf{B} \to \overline{\mathbf{B}}$ by sending $b$ to $\overline{b}$, such that $z \cdot \overline{b}=\overline{\overline{z} \cdot b}$ for all $z \in \mathbb{T}, b \in \mathbf{B}$. Define a principal circle bundle over $T_1 \times T_2$ by
\[
\mathbf{B}_1 \star \mathbf{B}_2:=(\mathbf{B}_1 \times \mathbf{B}_2)/\{(z \cdot b,b') \sim (b,z \cdot b'):b \in \mathbf{B}_1,b' \in \mathbf{B}_2,z \in \mathbb{T}\}.
\]
Inductively, for $n \geq 1$, we obtain a principal circle bundle $\mathbf{B}^{\star n}$ over $\prod_{i=1}^{n}T$. Notice that the restriction bundle of $\mathbf{B} \star \overline{\mathbf{B}}$ to $T$ is isomorphic to the product bundle $\mathbb{T} \times T$ by sending $(b,\overline{b'})$ to $(b/b',p(b))$.

\section{Boundary Paths}

Yeend in \cite{Yeend:JOT07, Yeend:CM06} gave a notion of boundary paths for topological $k$-graphs which include topological graphs. Webster in \cite{MR3119197} provided an alternative approach to define boundary paths of a directed graph. In this section we give a definition of boundary paths of a topological graph which is a generalization of Webster's definition, and we will prove that our definition of boundary paths of a topological graph coincides with Yeend's.

\begin{defn}\label{def of boundary path of top graph}
Let $E$ be a topological graph. Define the set of \emph{boundary paths} to be
\[
\partial E:=E^\infty \amalg \{\mu\in E^*:s(\mu) \in E_{\mathrm{sg}}^0 \}.
\]
\end{defn}

\begin{defn}[{\cite[Definitions~4.1, 4.2]{Yeend:JOT07}, \cite[Page~236]{Yeend:CM06}}]\label{define boundary paths of Yeend}
Let $E$ be a topological graph and let $V \subset E^0$. A set $U \subset r^{-1}(V) (\subset E^*)$ is said to be \emph{exhaustive} for $V$ if for any $\lambda \in r^{-1}(V)$ there exists $\alpha \in U$ such that $\lambda=\alpha\beta$ or $\alpha=\lambda\beta$.

An infinite path $\mu \in E^\infty$ is called a \emph{boundary path} in the sense of Yeend if for any $m \geq 0$, for any compact set $K \subset E^*$ such that $r(K)$ is a neighborhood of $r(\mu_{m+1})$ and $K$ is exhaustive for $r(K)$, there exists at least one path in the set $\{r(\mu_{m+1}), \mu_{m+1},\mu_{m+1}\mu_{m+2},\dots\}$ lying in $K$.

A finite path $\mu \in E^*$ is called a \emph{boundary path} in the sense of Yeend if for any $0 \leq m \leq \vert\mu\vert$, for any compact set $K \subset E^*$ such that $r(K)$ is a neighborhood of $r(\mu_{m+1})$ and $K$ is exhaustive for $r(K)$ if $m < \vert\mu\vert$, or that $r(K)$ is a neighborhood of $s(\mu)$ and $K$ is exhaustive for $r(K)$ if $m=\vert\mu\vert$, there exists at least one path in the set $\{r(\mu_{m+1}), \mu_{m+1},\dots,\mu_{m+1}\cdots\mu_{\vert\mu\vert}\}$ lying in $K$ if $0 \leq m < \vert \mu\vert$ or $s(\mu) \in K$ if $m=\vert\mu\vert$.

Denote by $\partial_Y E$ the set of all boundary paths in the sense of Yeend.
\end{defn}

\begin{rmk}\label{def of boundary paths in the sense of Yeend}
We explain Definition~\ref{define boundary paths of Yeend} in a more elementary way. Let $E$ be a topological graph and let $\mu \in E^* \amalg E^\infty$.

Let $\mu \in E^\infty$. Then $\mu \in \partial_Y E$ if and only if for $m \geq 0$, and for a compact subset $K \subset E^*$ satisfying both of the following conditions
\begin{enumerate}
\item\label{r(mu_{m+1}) in r(K)} $r(K)$ is a neighborhood of $r(\mu_{m+1})$,
\item\label{exhaustive condition} for $\lambda \in E^*$ with $r(\lambda) \in r(K)$ there exists $\alpha \in K$ such that $\lambda=\alpha\beta$ or $\alpha=\lambda\beta$,
\end{enumerate}
there exists at least one path in the set $\{r(\mu_{m+1}), \mu_{m+1},\mu_{m+1}\mu_{m+2},\dots\}$ lying in $K$.

Let $\mu \in E^*$. Then $\mu \in \partial_Y E$ if and only if for $0 \leq m \leq \vert \mu\vert$, for a compact subset $K \subset E^*$ satisfying both of the following conditions
\begin{enumerate}\setcounter{enumi}{2}
\item\label{r(mu_{m+1}) in r(K) finite path} $r(K)$ is a neighborhood of $r(\mu_{m+1})$ if $m < \vert \mu\vert$, or is a neighborhood of $s(\mu)$ if $m=\vert\mu\vert$,
\item\label{exhaustive condition finite path} for $\lambda \in E^*$ with $r(\lambda) \in r(K)$ there exists $\alpha \in K$ such that $\lambda=\alpha\beta$ or $\alpha=\lambda\beta$,
\end{enumerate}
there exists at least one path in the set $\{r(\mu_{m+1}), \mu_{m+1},\dots,\mu_{m+1}\cdots\mu_{\vert\mu\vert}\}$ lying in $K$ if $0 \leq m < \vert \mu\vert$, and $s(\mu) \in K$ if $m=\vert\mu\vert$.
\end{rmk}

\begin{lemma}\label{infinite paths are boundary paths}
Let $E$ be a topological graph. Fix $\mu \in E^\infty$. Then $\mu \in \partial_Y E$.
\end{lemma}
\begin{proof}
Fix $m \geq 0$, and fix a compact subset $K \subset E^*$ satisfying Conditions~(\ref{r(mu_{m+1}) in r(K)}), (\ref{exhaustive condition}) of Remark~\ref{def of boundary paths in the sense of Yeend}. Suppose that $r(\mu_{m+1}), \mu_{m+1},\mu_{m+1}\mu_{m+2},\dots \notin K$, for a contradiction. By Condition~(\ref{r(mu_{m+1}) in r(K)}) of Remark~\ref{def of boundary paths in the sense of Yeend}, $r(\mu_{m+1}) \in r(K)$. For $n \geq 1$, we have $r(\mu_{m+1}\cdots\mu_{m+n})=r(\mu_{m+1}) \in r(K)$. By Condition~(\ref{exhaustive condition}) of Remark~\ref{def of boundary paths in the sense of Yeend} and by the assumption, there exists $\beta^n \in E^* \setminus E^0$ such that $r(\beta^n)=s(\mu_{m+n})$ and $\alpha^n:=\mu_{m+1}\cdots\mu_{m+n}\beta^n \in K$. Thus we obtain a sequence of finite paths $(\alpha^n)_{n=1}^{\infty}$ contained in $K$ whose lengths are not bounded. However, the length of paths in $K$ is bounded since $K$ is compact in $E^*$. So we get a contradiction. Hence there exists at least one path in the set $\{r(\mu_{m+1}), \mu_{m+1},\mu_{m+1}\mu_{m+2},\dots\}$ lying in $K$. Therefore $\mu \in \partial_Y E$.
\end{proof}

\begin{lemma}\label{finite paths boundary paths}
Let $E$ be a topological graph. Fix $\mu \in E^*$. Then $\mu \in \partial E$ if and only if $\mu \in \partial_Y E$.
\end{lemma}
\begin{proof}
First of all, suppose that $\mu \in \partial E$. Then $s(\mu) \in E_{\mathrm{sg}}^0$. We split into two cases.

Fix $0 \leq m < \vert \mu\vert$, and fix a compact subset $K \subset E^*$ satisfying Conditions~(\ref{r(mu_{m+1}) in r(K) finite path}), (\ref{exhaustive condition finite path}) of Remark~\ref{def of boundary paths in the sense of Yeend}. Suppose that $r(\mu_{m+1}), \mu_{m+1},\dots,\mu_{m+1}\cdots\mu_{\vert\mu\vert} \notin K$, for a contradiction. There exist an open $s^{\vert\mu\vert-m}$-section $N$ of $\mu_{m+1}\cdots\mu_{\vert\mu\vert}$ and an open neighborhood $U$ of $s(\mu)$ such that

$\bullet$ $r^{\vert\mu\vert-m}(N) \subset r(K)$;

$\bullet$ for $\lambda \in N$, we have $r(\lambda), \lambda_{m+1},\dots,\lambda_{m+1}\cdots\lambda_{\vert\mu\vert} \notin K$; and

$\bullet$ $\overline{U} \subset s(N)$.

Case 1: $s(\mu) \notin E_{\mathrm{fin}}^0$. By Condition~(\ref{exhaustive condition finite path}) of Remark~\ref{def of boundary paths in the sense of Yeend} for any net $(e_a)_{a \in A} \subset r^{-1}(\overline{U})$, there exist a net $(\lambda^a)_{a \in A} \subset N$ and a net $(\beta^a)_{a \in A} \subset E^*$, such that $\lambda^a e_a \beta^a$ is a path for $a \in A$, and $(\lambda^a e_a \beta^a)_{a \in A} \in K$. So there exists a convergent subnet of the net $(e_a)_{a \in A}$ because $K$ is compact. Since $(e_a)_{a \in A}$ is arbitrary, $r^{-1}(\overline{U})$ is compact. On the other hand, since $s(\mu) \notin E_{\mathrm{fin}}^0, r^{-1}(\overline{U})$ is then not compact. Hence we deduce a contradiction. Therefore there exists at least one path in the set $\{r(\mu_{m+1}), \mu_{m+1},\dots,\mu_{m+1}\cdots\mu_{\vert\mu\vert}\}$ lying in $K$.

Case 2: $s(\mu) \in \overline{E^0 \setminus \overline{r(E^1)}}$. Since $s(N)$ is an open neighborhood of $s(\mu)$, there exists $v \in s(N) \setminus \overline{r(E^1)}$. Then there exists $\lambda \in N$ such that $s(\lambda)=v$. So $r(\lambda), \lambda_{m+1}, \dots$, and $\lambda_{m+1}\dots\lambda_{\vert\mu\vert} \notin K$. However, since $v \notin r(E^1)$, there exists at least one path in the set $\{r(\mu_{m+1}), \mu_{m+1},\dots,\mu_{m+1}\cdots\mu_{\vert\mu\vert}\}$ lying in $K$, which is a contradiction. Hence there exists at least one path in the set $\{r(\mu_{m+1}), \mu_{m+1},\dots$, $\mu_{m+1}\cdots\mu_{\vert\mu\vert}\}$ lying in $K$.

Now fix $m = \vert \mu\vert$, and fix a compact subset $K \subset E^*$ satisfying Conditions~(\ref{r(mu_{m+1}) in r(K) finite path}), (\ref{exhaustive condition finite path}) of Remark~\ref{def of boundary paths in the sense of Yeend}. Similar arguments as above yield that $s(\mu) \in K$. So $\mu \in \partial_Y E$.

Conversely, suppose that $\mu \in \partial_Y E$. Suppose that $s(\mu) \in E_{\mathrm{rg}}^0$, for a contradiction. By \cite[Proposition~2.8]{Katsura:TAMS04}, there exists a neighborhood $N$ of $s(\mu)$ such that $r^{-1}(N)$ is compact and $r(r^{-1}(N))=N$. Let $m=\vert\mu\vert$ and let $K=r^{-1}(N)$. It is straightforward to check that $K$ satisfies Conditions~(\ref{r(mu_{m+1}) in r(K) finite path}), (\ref{exhaustive condition finite path}) of Remark~\ref{def of boundary paths in the sense of Yeend}. By the assumption, we get $s(\mu) \in K$, but this is impossible because $K \subset E^1$. So we deduce a contradiction. Hence $s(\mu) \in E_{\mathrm{sg}}^0$ and $\mu \in \partial E$.
\end{proof}

\begin{prop}\label{def of boundary path coincide with Yeend's}
Let $E$ be a topological graph. Then $\partial E=\partial_Y E$.
\end{prop}
\begin{proof}
It follows immediately from Lemmas~\ref{infinite paths are boundary paths}, \ref{finite paths boundary paths}.
\end{proof}

Let $E$ be a topological graph and let $\mu \in E^* \amalg E^\infty$. From now on, whenever we say $\mu$ is a boundary path we mean that $\mu$ is a boundary path in the sense of Definition~\ref{def of boundary path of top graph} unless specified otherwise.

Since the product topology on $E^\infty$ may not be locally compact in general it is not obvious how to endow  the boundary path space $\partial E$ with a locally compact Hausdorff topology.
In \cite{Yeend:JOT07, Yeend:CM06} Yeend defined such a topology on the boundary path space of a topological higher rank graph.
So using the identification of Proposition~\ref{def of boundary path coincide with Yeend's},
we can endow the boundary path space $\partial E$ with the locally compact Hausdorff topology used by Yeend.

The following definition is a slight modification of \cite[Proposition~3.6]{Yeend:JOT07} for topological graphs.

\begin{defn}\label{def top on partial E}
Let $E$ be a topological graph. For a subset $S \subset E^*$, denote by $Z(S):=\{\mu \in \partial E: \text{ either } r(\mu) \in S, \text{ or there exists } 1 \leq i \leq \vert\mu\vert, \text{ such that }$ $\mu_1 \cdots\mu_i \in S\}$. We endow $\partial E$ with the topology generated by the basic open sets $Z(U) \cap Z(K)^c$, where $U$ is an open set of $E^*$ and $K$ is a compact set of $E^*$.
\end{defn}

It follows now using the identification of $\partial_YE$ with $\partial E$ above that $\partial E$ is a locally compact Hausdorff space. One verifies that $E_{\mathrm{sg}}^0$ is a closed subset of $\partial E$; that $Z(U)$ is open for every open subset $U \subset E^*$; and that $Z(K)$ is compact for every compact subset $K \subset E^*$.

\begin{lemma}\label{convergent net in par E}
Let $E$ be a topological graph. Fix a sequence $(\mu^{(n)})_{n=1}^\infty \subset \partial E$, and fix $\mu\in \partial E$. Then $\mu^{(n)} \to \mu$ if and only if
\begin{enumerate}
\item\label{r(mu^n) to r(mu)} $r(\mu^{(n)}) \to r(\mu)$;
\item\label{mu^n_i to mu_i} for $1 \leq i \leq \vert\mu\vert$ with $i\neq \infty$, there exists $N \geq 1$ such that $\vert\mu^{(n)}\vert \geq i$ whenever $n \geq N$ and $(\mu^{(n)}_1  \cdots \mu^{(n)}_i)_{n \geq N} \to \mu_1 \cdots\mu_i$;
\item\label{wandering condition} if $\vert \mu\vert < \infty$, then for any compact set $K \subset E^1$, the set $\{n: \vert\mu^{(n)}\vert > \vert\mu\vert \text{ and }$ $\mu^{(n)}_{\vert\mu\vert+1} \in K\}$ is finite.
\end{enumerate}
\end{lemma}
\begin{proof}
Suppose that $\mu^{(n)} \to \mu$. Conditions~(\ref{r(mu^n) to r(mu)})--(\ref{mu^n_i to mu_i}) are straightforward to verify. Suppose that $\vert\mu\vert<\infty$. We may assume that $\vert\mu\vert\geq 1$. Fix a compact set $K \subset E^1$. Take a precompact neighborhood $U$ of $\mu$ in $E^{\vert\mu\vert}$. Then $\mu \in Z(U) \cap Z((\overline{U} \times K) \cap E^{\vert\mu\vert+1})^c$. Since $\mu^{(n)} \to \mu$, there exists $N \geq 1$ such that $\mu^{(n)} \in Z(U) \cap Z((\overline{U} \times K) \cap E^{\vert\mu\vert+1})^c$ whenever $n \geq N$. So the set $\{n: \vert\mu^{(n)}\vert > \vert\mu\vert \text{ and } \mu^{(n)}_{\vert\mu\vert+1} \in K\}$ is finite.

Conversely, suppose that Conditions~(\ref{r(mu^n) to r(mu)})--(\ref{wandering condition}) hold. Fix an open neighborhood $Z(U) \cap Z(K)^c$ of $\mu$.

Case $1$: $\vert\mu\vert=\infty$. It is straightforward to check that there exists $N \geq 1$ such that $\mu^{(n)} \in Z(U)$ whenever $n \geq N$. Since $\mu\in Z(K)^c$, we have $r(\mu), \mu_1,\mu_1\mu_2,\dots \notin K$. Conditions~(\ref{r(mu^n) to r(mu)}), (\ref{mu^n_i to mu_i}) imply that there exists $N' \geq N$ such that $\mu^{(n)} \in Z(K)^c$. So $\mu^{(n)} \to \mu$.

Case $2$: $\vert\mu\vert < \infty$. We may assume that $\vert\mu\vert\geq 1$. It is straightforward to check that there exists $N \geq 1$ such that $\vert\mu^{(n)}\vert\geq\vert\mu\vert, \mu^{(n)} \in Z(U)$ whenever $n \geq N$. Suppose that $K \cap (\amalg_{i =\vert\mu\vert+1}^{\infty}E^i)=\emptyset$. Then Conditions~(\ref{r(mu^n) to r(mu)}), (\ref{mu^n_i to mu_i}) imply that there exists $N' \geq N$ such that $\mu^{(n)} \in Z(K)^c$ whenever $n \geq N'$. Suppose that $K \cap (\amalg_{i =\vert\mu\vert+1}^{\infty}E^i)\neq\emptyset$. Then the set $K':=\{\nu_{\vert\mu\vert+1}:\nu \in K \cap (\amalg_{i =\vert\mu\vert+1}^{\infty}E^i)\}$ is a compact set of $E^1$. Since the set $\{n: \vert\mu^{(n)}\vert > \vert\mu\vert \text{ and } \mu^{(n)}_{\vert\mu\vert+1} \in K'\}$ is finite by Condition~\ref{wandering condition}, we deduce that there exists $N'' \geq N$ such that $\mu^{(n)} \in Z(K)^c$ whenever $n \geq N''$.
\end{proof}

It follows from Lemma~\ref{convergent net in par E} and \cite[Proposition~3.12]{Yeend:JOT07} that the topology on the boundary path space given in Definition~\ref{def top on partial E} agrees with the topology on the boundary path space given in \cite[Proposition~3.6]{Yeend:JOT07}.

\section{Factor Maps}

In this section, we recall the notion of factor maps between topological graphs introduced by Katsura in \cite[Section~2]{Katsura:IJM06}. Our definition of factor maps is a special case of Katsura's (see Remark~\ref{special case of Katsura's factor map}).

\begin{defn}
Let $E=(E^0,E^1,r_E,s_E), F=(F^0,F^1,r_F,s_F)$ be topological graphs and let $m^0:F^0 \to E^0, m^1:F^1 \to E^1$ be proper continuous maps. Then the pair $m:=(m^0,m^1)$ is called a \emph{factor map} from $F$ to $E$ if
\begin{enumerate}
\item $r_E \circ m^1=m^0 \circ r_F, s_E \circ m^1=m^0 \circ s_F$; and
\item for $e \in E^1, u \in F^0$, if $s_E(e)=m^0(u)$, then there exists a unique $f \in F^1$, such that $m^1(f)=e, s_F(f)=u$.
\end{enumerate}
Moreover, the factor map is called \emph{regular} if $m^0(F_{\mathrm{sg}}^0) \subset E_{\mathrm{sg}}^0$.
\end{defn}

\begin{rmk}\label{special case of Katsura's factor map}
By \cite[Lemma~2.7]{Katsura:IJM06}, we are able to give some equivalent conditions under which factor maps are regular. The factor map is regular if and only if $(m^0)^{-1}(E_{\mathrm{rg}}^0) \subset F_{\mathrm{rg}}^0$ if and only if for any $u \in F^0$ with $m^0(u) \in E_{\mathrm{rg}}^0$, we have $r_F^{-1}(u) \neq \emptyset$.
\end{rmk}

\begin{rmk}
Our definition of factor maps is indeed a special case of the one defined by Katsura in \cite{Katsura:IJM06}. In our case, we can extend $m^0$ continuously to the one-point compactification of $F^0$ by sending $\infty$ to $\infty$, and extend $m^1$ in the same way. Then we get a factor map in the sense of \cite[Definitions~2.1, 2.6]{Katsura:IJM06}.
\end{rmk}

The proofs of the following two propositions are similar to \cite[Propositions~2.9, 2.10]{Katsura:IJM06}. Consequently we just state these results without proofs.

\begin{prop}\label{homo bet twisted top graph alg}  
Let $E=(E^0,E^1,r_E,s_E), F=(F^0,F^1,r_F,s_F)$ be topological graphs, let $m:=(m^0,m^1)$ be a regular factor map from $F$ to $E$, and let $p_E:\mathbf{B}_E \to E^1$ be a principal circle bundle over $E^1$. Denote by $p_F:\mathbf{B}_F \to F^1$ the principal circle bundle which is the pullback of $\mathbf{B}_E$ by $m^1$. Denote by $m_*^1: X(E,\mathbf{B}_E) \to X(F,\mathbf{B}_F)$ the induced linear map from $m^1$, and denote by $m_*^0:C_0(E^0) \to C_0(F^0)$ the induced homomorphism from $m^0$. Let $(j_{X,E},j_{A,E})$ be the universal covariant  representation of $X(E,\mathbf{B}_E)$ into $\mathcal{O}(E,\mathbf{B}_E)$, and let $(j_{X,F},j_{A,F})$ be the universal covariant representation of $X(F,\mathbf{B}_F)$ into $\mathcal{O}(F,\mathbf{B}_F)$. Then $(j_{X,F} \circ m_*^1,j_{A,F} \circ m_*^0)$ is a covariant  representation of $X(E,\mathbf{B}_E)$ into $\mathcal{O}(F,\mathbf{B}_F)$. Hence there exists a unique homomorphism $h:\mathcal{O}(E,\mathbf{B}_E) \to \mathcal{O}(F,\mathbf{B}_F)$ such that $h \circ j_{X,E}=j_{X,F} \circ m_*^1, h \circ j_{A,E}=j_{A,F} \circ m_*^0$. Moreover, $h$ is injective if and only if $m^0$ is surjective.
\end{prop}

\begin{prop}
Let $E=(E^0,E^1,r_E,s_E), F=(F^0,F^1,r_F,s_F), G=(G^0,G^1$, $r_G,s_G)$ be topological graphs, let $m=(m^0,m^1)$ be a regular factor map from $F$ to $E$, let $n=(n^0,n^1)$ be a regular factor map from $G$ to $F$, and let $p_E:\mathbf{B}_E \to E^1$ be a principal circle bundle. Denote by $p_F:\mathbf{B}_F \to F^1$ the principal circle bundle which is the pullback of $\mathbf{B}_E$ by $m_1$, and denote by $p_G:\mathbf{B}_G \to G^1$ the principal circle bundle which is the pullback of $\mathbf{B}_E$ by $m_1 \circ n_1$. We have the following.
\begin{enumerate}
\item $m \circ n:=(m^0 \circ n^0,m^1 \circ n^1)$ is a regular factor map from $G$ to $E$.
\item Let $h_1: \mathcal{O}(E,\mathbf{B}_E) \to \mathcal{O}(F,\mathbf{B}_F)$ be the homomorphism induced from the regular factor map $m$, let $h_2: \mathcal{O}(F,\mathbf{B}_F) \to \mathcal{O}(G,\mathbf{B}_G)$ be the homomorphism induced from $n$, and let $h_3: \mathcal{O}(E,\mathbf{B}_E) \to \mathcal{O}(G,\mathbf{B}_G)$ be the homomorphism induced from $m \circ n$. Then $h_3=h_2 \circ h_1$.
\end{enumerate}
\end{prop}

\section{Twisted Groupoid $C^*$-algebras}

In this section, we deal with groupoids and groupoid $C^*$-algebras (see \cite{Renault:groupoidapproachto80}).

From now on we assume that all the topological spaces are second countable; and that all the locally compact groupoids are second-countable locally compact Hausdorff groupoids. A locally compact groupoid is said to be \emph{\'{e}tale} if its range map is a local homeomorphism.

\begin{defn}[{\cite[Remark~2.9]{Kumjian:JOT88}}]\label{defn of top twist}
Let $\Gamma$ be an \'{e}tale groupoid, and let $\Lambda$ be a locally compact groupoid. Suppose that $\Gamma$ and $\Lambda$ have a common unit space $\Gamma^0$. We call $\Lambda$ a \emph{topological twist} over $\Gamma$ if there is a sequence of groupoid homomorphisms
\[
\mathbb{T} \times \Gamma^0 \xrightarrow{i} \Lambda \xrightarrow{p} \Gamma
\]
such that
\begin{enumerate}
\item\label{i(T times Gamma^0)=p^{-1}(Gamma^0)} $i$ is a homeomorphism onto $p^{-1}(\Gamma^0)$;
\item $p$ is a continuous open surjection and admits continuous local sections; and
\item\label{condition 3} $\lambda i(z,s(\lambda)) \lambda^{-1}=i(z,r(\lambda))$, for all $z \in \mathbb{T}$, and all $\lambda \in \Lambda$.
\end{enumerate}
\end{defn}

By \cite[Remark~2.9]{Kumjian:JOT88}, we are able to define a free and proper circle action on $\Lambda$ by $z \cdot \lambda:=i(z,r(\lambda))\lambda$. The quotient space $\Lambda / \mathbb{T}$ is homeomorphic to $\Gamma$ via the identification map $[\lambda] \mapsto p(\lambda)$. Since $p$ admits continuous local sections, $p:\Lambda \to \Gamma$ can be regarded as a principal circle bundle. For $u \in \Gamma^0$, we have $r^{-1}(u)$ is a discrete subset of $\Gamma$ because $r:\Gamma \to \Gamma^0$ is a local homeomorphism. Since $p:\Lambda \to \Gamma$ is a principal circle bundle and since $r^{-1}(u)=p^{-1}(r^{-1}(u))$, we get $r^{-1}(u)$ is a disjoint union of circles. So there is a natural measure on $r^{-1}(u)$ and $\Lambda$ has a left Haar system $\{\mu^u\}_{u \in \Gamma^0}$ (see \cite[Page~252]{Kumjian:JOT88}).

\begin{defn}\cite[Page~252]{Kumjian:JOT88}\label{define the twisted groupoid C^*-alg}
Let $\Gamma$ be an \'{e}tale groupoid and fix a topological twist over $\Gamma$
\[
\mathbb{T} \times \Gamma^0 \xrightarrow{i} \Lambda \xrightarrow{p} \Gamma.
\]
The closure of $\{f \in C_c(\Lambda): f(z \cdot \lambda)=zf(\lambda) \text{ for all } z \in \mathbb{T}\}$ under the $C^*$-norm of the groupoid $C^*$-algebra $C^*(\Lambda)$ is called the \emph{twisted groupoid $C^*$-algebra} and is denoted by $C^*(\Gamma,\Lambda)$.
\end{defn}

The convolution product (see \cite[Page~48]{Renault:groupoidapproachto80}) of $C^*(\Gamma,\Lambda)$ is given as follows. For $f,g \in \{f \in C_c(\Lambda): f(z \cdot \lambda)=zf(\lambda) \text{ for all } z \in \mathbb{T}\}$, we have
\begin{align*}
f*g(\lambda)&=\int_{\lambda' \in r^{-1}(s(\lambda))} f(\lambda\lambda')g(\lambda'^{-1})\, \mathrm{d}\mu^{s(\lambda)}(\lambda')\\
&=\sum_{\gamma \in r^{-1}(s(\lambda))}\int_{\lambda' \in p^{-1}(\gamma)} f(\lambda\lambda')g(\lambda'^{-1})
\intertext{note that $f(\lambda\lambda')g(\lambda'^{-1})$ is constant on each fibre $p^{-1}(\gamma)$ and so}
&=\sum_{\gamma \in r^{-1}(s(\lambda))} f(\lambda\lambda_\gamma)g(\lambda_\gamma^{-1}).
\end{align*}
where $\gamma \mapsto \lambda_\gamma$ is any section of $p$.
\begin{rmk}\label{C*(Gamma,Lambda) contains C_0(Gamma^0)}
It follows from \cite{DeaconuKumjianEtAl:JOT01, Renault:IMSB08} that there is an injective homomorphism $\pi:C_0(\Gamma^0) \to C^*(\Gamma,\Lambda)$ such that for $h \in C_c(\Gamma^0), \pi(h)=\widetilde{h}$, where
\[
\widetilde{h}(\lambda):= \begin{cases} zh(t), &\text{ if $(z,t) \in \mathbb{T} \times \Gamma^0, \lambda=i(z,t)$;}
   \\ 0, &\text{ if $\lambda \notin p^{-1}(\Gamma^0)$.} \end{cases}
\]
\end{rmk}

Now we start to look at the groupoid induced from a singly generated dynamical system (see Page~2) and investigate its topological twists.

\begin{defn}[{\cite[Definition~2.4]{Renault:00}}]\label{define D-R groupoid}
Let $T$ be a locally compact Hausdorff space and let $\sigma:\dom(\sigma) \to \ran(\sigma)$ be a partial local homeomorphism (see Page~2). Define the \emph{Renault-Deaconu groupoid} $\Gamma(T,\sigma)$ as follows:
\begin{align*}
\Gamma(T,\sigma):=\{(t_1,k_1-k_2,t_2) \in T \times \mathbb{Z} \times T : \ & k_1,k_2 \geq 0,t_1 \in \dom(\sigma^{k_1}),
\\&t_2 \in \dom(\sigma^{k_2}), \sigma^{k_1}(t_1)=\sigma^{k_2}(t_2) \}.
\end{align*}
Define the unit space $\Gamma^0:=\{(t,0,t):t \in T\}$. For $(t_1,n,t_2),(t_2,m,t_3) \in \Gamma(T,\sigma)$, define the multiplication, the inverse, the source and the range map by
\[
(t_1,n,t_2)(t_2,m,t_3):=(t_1,n+m,t_3); \ \ \ \ (t_1,n,t_2)^{-1}:=(t_2,-n,t_1);
\]
\[
 r(t_1,n,t_2):=(t_1,0,t_1); \ \ \ \  s(t_1,n,t_2):=(t_2,0,t_2).
\]
Define the topology on $\Gamma(T,\sigma)$ to be generated by the basic open set
\[
\mathcal{U}(U,V,k_1,k_2):=\{(t_1,k_1-k_2,t_2):t_1 \in U,t_2 \in V,\sigma^{k_1}(t_1)=\sigma^{k_2}(t_2)\},
\]
where $U \subset \dom(\sigma^{k_1}), V \subset \dom(\sigma^{k_2})$ are open in $T,\sigma^{k_1}$ is injective on $U$, and $\sigma^{k_2}$ is injective on $V$. For $k_1,k_2 \geq 0$, define an open subset of $\Gamma(T,\sigma)$ by
\[
\Gamma_{k_1,k_2}:=\{(t_1,k_1-k_2,t_2): t_1 \in \dom(\sigma^{k_1}),t_2 \in \dom(\sigma^{k_2}), \sigma^{k_1}(t_1)=\sigma^{k_2}(t_2)\}.
\]
\end{defn}

The Renault-Deaconu groupoid $\Gamma(T,\sigma)$ is an \'{e}tale groupoid.

We give the characterization of convergent nets in $\Gamma(T,\sigma)$. Fix $((t_{1,\alpha},n_\alpha$, $t_{2,\alpha}))_{\alpha \in A}$ in $\Gamma(T,\sigma)$, and fix $(t_1,n,t_2) \in \Gamma(T,\sigma)$. Find $k_1,k_2 \geq 0$ such that
\begin{enumerate}
\item $n=k_1-k_2, t_1 \in \dom(\sigma^{k_1}),t_2 \in \dom(\sigma^{k_2}), \sigma^{k_1}(t_1)=\sigma^{k_2}(t_2)$; and that
\item if there exist $k_1',k_2' \geq 0$ satisfying that $k_1' \leq k_1, k_2' \leq k_2, n=k_1'-k_2'$, $\sigma^{k_1'}(t_1)=\sigma^{k_2'}(t_2)$, then we have $k_1'=k_1, k_2'=k_2$.
\end{enumerate}
we have $(t_{1,\alpha},n_\alpha,t_{2,\alpha}) \to (t_1,n,t_2)$ if and only if $t_{1,\alpha} \to t_1, t_{2,\alpha} \to t_2$, and there exists $\alpha_0 \in A$ such that whenever $\alpha \geq \alpha_0$, we have $n_\alpha=n, t_{1,\alpha} \in \dom(\sigma^{k_1}),t_{2,\alpha} \in \dom(\sigma^{k_2}), \sigma^{k_1}(t_{1,\alpha})=\sigma^{k_2}(t_{2,\alpha})$.

\begin{lemma}\label{piece principal bundles open version}
Let $Z$ be a locally compact Hausdorff space, let $\{Z_n\}_{n \geq 1}$ be a countable open cover of $Z$, and let $\{ p_n:\mathbf{B}_n \to Z_n \}$ be a family of principal circle bundles. Suppose that for $n,m \geq 1$, there exists a homeomorphism $h_{n,m}:p_n^{-1}(Z_n \cap Z_m) \to p_m^{-1}(Z_n \cap Z_m)$ such that $p_m \circ h_{n,m}=p_n,h_{n,m}(z \cdot b)=z \cdot h_{n,m}(b)$ for all $z \in \mathbb{T}, b \in p_n^{-1}(Z_n \cap Z_m)$, and $h_{m,l} \circ h_{n,m}=h_{n,l}$ on $p_n^{-1}(Z_n \cap Z_m \cap Z_l)$. Define
\[
\mathbf{B}:=\amalg_{n \geq 1}\mathbf{B}_n/ \{(b,n) \sim (h_{n,m}(b),m): b \in p_n^{-1}(Z_n \cap Z_m)\}
\]
endowed with the quotient topology. For $N \geq 1$, for a sequence $(b_i,N)_{i=1}^{\infty} \subset \mathbf{B}$, and for $(b,N) \in \mathbf{B}$, we have $(b_i,N) \to (b,N)$ in $\mathbf{B}$ if and only if $b_i \to b$ in $\mathbf{B}_N$. Moreover, $\mathbf{B}$ is a (second-countable) principal circle bundle over $Z$.
\end{lemma}
\begin{proof}
It is straightforward to verify.
\end{proof}
Next we generalize  \cite[Theorem~3.1]{DeaconuKumjianEtAl:JOT01} so that it applies to partial local homeomorphisms and not just local homeomorphisms.  The proof is similar.
\begin{thm}\label{1-1 corr bet twisted groupoid and line bundle}
Let $T$ be a locally compact Hausdorff space, let $\sigma:\dom(\sigma) \to \ran(\sigma)$ be a partial local homeomorphism, and let $p:\mathbf{B} \to \dom(\sigma)$ be a principal circle bundle. Denote by $j:\dom(\sigma) \to \Gamma(T,\sigma)$ the embedding such that $j(t)=(t,1,\sigma(t))$.
Then there exists a topological twist $\mathbb{T} \times \Gamma^0 \xrightarrow{i} \Lambda \xrightarrow{p'} \Gamma(T,\sigma)$, such that the pullback bundle $j^*(\Lambda)$ of $\Lambda$ by $j$ is isomorphic to $\mathbf{B}$.
\end{thm}
\begin{proof}
For $k_1,k_2 \geq 1$, we have a principal circle bundle $\mathbf{B}^{\star k_1} \star \overline{\mathbf{B}}^{\star k_2}$ over $(\prod_{i=1}^{k_1}\dom(\sigma)) \times (\prod_{j=1}^{k_2}\dom(\sigma))$. Denote by $\iota_{k_1,k_2}:\Gamma_{k_1,k_2} \to (\prod_{i=1}^{k_1}\dom(\sigma)) \times (\prod_{j=1}^{k_2}\dom(\sigma))$ the embedding
\[
\iota(t_1,k_1-k_2,t_2):=(t_1, \sigma(t_1), \dots,\sigma^{k_1-1}(t_1),t_2, \sigma(t_2), \dots,\sigma^{k_2-1}(t_2)).
\]
Denote by $p_{k_1,k_2}:\Lambda_{k_1,k_2} \to\Gamma_{k_1,k_2}$ the pullback bundle of $\mathbf{B}^{\star k_1} \star \overline{\mathbf{B}}^{\star k_2}$ by $\iota_{k_1,k_2}$.

For $k \geq 1$, there are embeddings $\iota_{k,0}:\Gamma_{k,0} \to \prod_{i=1}^{k}\dom(\sigma),\ \iota_{0,k}:\Gamma_{0,k} \to \prod_{i=1}^{k}\dom(\sigma)$, and similarly we get principal circle bundles $\Lambda_{k,0}$ over $\Gamma_{k,0}$ and  $\Lambda_{0,k}$ over $\Gamma_{0,k}$.

Moreover, we may identify $\Gamma_{0,0}$ with $T$ via the homeomorphism $\iota_{0,0}:\Gamma_{0,0} \to T$.
Denote by $\Lambda_{0,0}$ the trivial principal circle bundle $\mathbb{T} \times T$ over $T$.

For $k_1, k_2 \geq 1$, define $h_{(k_1,k_2),(k_1,k_2)}:=\id$.

For $1 \leq k_1< l_1,1 \leq k_2 < l_2$ with $k_1-k_2=l_1-l_2$, define
\[
h_{(l_1,l_2),(k_1,k_2)}:p_{l_1,l_2}^{-1}(\Gamma_{k_1,k_2} \cap \Gamma_{l_1,l_2}) \to p_{k_1,k_2}^{-1}(\Gamma_{k_1,k_2} \cap \Gamma_{l_1,l_2})
\]
as follows. For $(b_1,\dots,b_{l_1},\overline{b_1'},\dots,\overline{b_{l_2}'}) \in p_{l_1,l_2}^{-1}(\Gamma_{k_1,k_2} \cap \Gamma_{l_1,l_2})$, define
\begin{align*}
h_{(l_1,l_2),(k_1,k_2)}(b_1,\dots,b_{l_1},\overline{b_1'},\dots,\overline{b_{l_2}'}):=&(b_{k_1+1}/b_{k_2+1}') \cdots( b_{l_1}/b_{l_2}')\\&(b_1,\dots,b_{k_1},\overline{b_1'},\dots,\overline{b_{k_2}'}).
\end{align*}
It is routine to show that $h_{(l_1,l_2),(k_1,k_2)}$ is a homeomorphism;  its inverse
is given by
\[
h_{(k_1,k_2), (l_1,l_2)}(b_1,\dots,b_{k_1},\overline{b_1'},\dots, \overline{b_{k_2}'}) =
(b_1, \dots, b_{k_1}, c_1, \dots, c_j,
\overline{b_1'},\dots, \overline{b_{l_1}'}, \overline{c_1}, \dots, \overline{c_j})
\]
where $j := l_1 - k_1 = l_2 - k_2$; note that the formula does not depend on the choice of the $c_i$.
The formulas above give homeomorphisms for all $k_1,k_2,l_1,l_2 \geq 0$
with $k_1-k_2=l_1-l_2$.

It is straightforward to check that for $k_1,k_2,l_1,l_2,m_1,m_2 \geq 0$ with $k_1-k_2=l_1-l_2=m_1-m_2$, we have $p_{l_1,l_2} \circ h_{(k_1,k_2),(l_1,l_2)}=p_{k_1,k_2}$, and $h_{(l_1,l_2),(m_1,m_2)} \circ h_{(k_1,k_2),(l_1,l_2)}=h_{(k_1,k_2),(m_1,m_2)}$ on $p_{k_1,k_2}^{-1}(\Gamma_{k_1,k_2} \cap \Gamma_{l_1,l_2} \cap \Gamma_{m_1,m_2})$.
By Lemma~\ref{piece principal bundles open version}, we may construct a locally compact Hausdorff space $\Lambda_n$ for
$n \in \mathbb{Z}$ by

\[
\Lambda_n:=
\coprod_{\substack{ k_1,k_2 \geq 0\\ k_1-k_2= n}} \Lambda_{k_1,k_2}/ \sim
\]
where $\lambda \sim h_{(k_1,k_2),(l_1,l_2)}(\lambda)$
for all $\lambda \in p_{k_1,k_2}^{-1}(\Gamma_{k_1,k_2} \cap \Gamma_{l_1,l_2})\}$.
For $k_1,k_2,l_1,l_2 \geq 0$, if $k_1-k_2 \neq l_1-l_2$, then $\Gamma_{k_1,k_2} \cap \Gamma_{l_1,l_2}=\emptyset$. Observe that  $\Lambda:=\coprod_{n \in \mathbb{Z}}\Lambda_n$ is a locally compact Hausdorff space which we may view as a circle bundle over $\Gamma(T,\sigma)$ with bundle map $p':\Lambda \to \Gamma(T,\sigma)$ defined in the obvious way ($p'([\lambda]) = p_{k_1,k_2}(\lambda)$ where $\lambda \in \Lambda_{k_1,k_2}$).

Now we endow $\Lambda$ with a groupoid structure.
We define the range and source maps $r, s : \Lambda \to \Gamma^0$
$r(\lambda) = r_\Lambda(\lambda) = r(p'(\lambda))$ and $s(\lambda) = s_\Lambda(\lambda) = s(p'(\lambda))$
for $\lambda \in \Lambda$.
Now let $\lambda_1, \lambda_2 \in \Lambda$ such that $s(\lambda_1) = r(\lambda_2)$.
Then there exist
$k_i \geq 1$, for $i = 1, 2, 3$,
$(b_1,\dots,b_{k_1},\overline{b_1'},\dots,\overline{b_{k_2}'}) \in \Lambda_{k_1,k_2}$ and
$(b_1'',\dots,b_{k_2}'',\overline{b_1'''},\dots,\overline{b_{k_2}'''}) \in \Lambda_{k_2,k_3}$
such that $\lambda_1 = [(b_1,\dots,b_{k_1},\overline{b_1'},\dots,\overline{b_{k_2}'})]$ and
$\lambda_2 = [(b_1'',\dots,b_{k_2}'',\overline{b_1'''},\dots,\overline{b_{k_2}'''})]$ and
$p(b_1')=p(b_1'')$.
Define
\begin{align*}
\lambda_1\lambda_2 &=
[(b_1,\dots,b_{k_1},\overline{b_1'},\dots,\overline{b_{k_2}'})] \cdot
 [(b_1'',\dots,b_{k_2}'',\overline{b_1'''},\dots,\overline{b_{k_3}'''})] \\
  &:=[(b_1''/b_1')\cdots(b_{k_2}''/b_{k_2}')(b_1,\dots,b_{k_1},\overline{b_1'''},\dots,\overline{b_{k_3}'''})];
\end{align*}
and
\[
(b_1,\dots,b_{k_1},\overline{b_1'},\dots,\overline{b_{k_2}'})^{-1}:=(b_1',\dots,b_{k_2}',\overline{b_1},\dots,\overline{b_{k_1}}).
\]
It is straightforward to check that $\Lambda$ is a locally compact groupoid under these two operations with the unit space $\Lambda^0$ which is homeomorphic to $\Gamma^0$.

Define $i:\Gamma^0 \times \mathbb{T} \to \Lambda$ to be the embedding such that its image is $\Lambda_{0,0}$. Define $p':\Lambda \to \Gamma(T,\sigma)$ in the obvious way. Then Conditions~(\ref{i(T times Gamma^0)=p^{-1}(Gamma^0)})--(\ref{condition 3}) of Definition~\ref{defn of top twist} follow.

The rest of the proof is straightforward.
\end{proof}

By arguing along the lines of \cite[Theorem~3.1]{DeaconuKumjianEtAl:JOT01} it can be shown that
the topological twist $\Lambda$ in the above theorem is unique.

The following theorem is a generalization of \cite[Theorem~3.3]{DeaconuKumjianEtAl:JOT01}. In particular, we consider partial local homeomorphisms instead of local homeomorphisms.

\begin{thm}\label{The twisted top graph alg is iso to the twisted gpoid C*-alg}
Let $T$ be a locally compact Hausdorff space and let $\sigma:\dom(\sigma) \to \ran(\sigma)$ be a partial local homeomorphism. Define a topological graph $E:=(T,\dom(\sigma)$, $\id,\sigma)$. Fix a topological twist \[
\mathbb{T} \times \Gamma^0 \xrightarrow{i} \Lambda \xrightarrow{p} \Gamma(T,\sigma).
\]
Denote $\mathbf{B}:=j^*(\Lambda)$. Then the twisted topological graph algebra $\mathcal{O}(E,\mathbf{B})$ is isomorphic to the twisted groupoid $C^*$-algebra $C^*(\Gamma(T,\sigma),\Lambda)$.
\end{thm}
\begin{proof}
Denote $Q:\mathbb{T} \times \Gamma^0 \to \mathbb{T}$ the natural projection.
We may identify $\mathbf{B} = j^*(\Lambda)$ with $(p')^{-1}(j(\dom \sigma))$
which is a clopen subset of $\Lambda$.
Let $x$ be an equivariant complex-valued continuous function with compact support on $\mathbf{B}$;
then using the above identification and extending by zero yields an equivariant complex-valued continuous function with compact support on $\Lambda$ which we denote by $\psi(x)$.
It is straightforward to check that this yields a linear map $\psi$.
Let $\pi: C_0(T) \to C^*(\Gamma(T,\sigma),\Lambda)$ be the injective homomorphism as described in Remark~\ref{C*(Gamma,Lambda) contains C_0(Gamma^0)}.


Fix two equivariant complex-valued continuous function with compact support $x, y$ on $\mathbf{B}$, fix $h \in C_c(T)$, and fix $\lambda \in \Lambda$.
Let $\lambda \in \mathbf{B}$ and write $p(\lambda)=(t,1,\sigma(t))$. Then
\begin{align*}
\pi(h)*\psi(x)(\lambda)&=\pi(h)(\lambda\lambda^{-1})\psi(x)(\lambda)
\\&=Q \circ i^{-1}(\lambda\lambda^{-1})h(p(\lambda\lambda^{-1}))x(\lambda)
\\&=\psi(h \cdot x)(\lambda).
\end{align*}
So $\psi(h \cdot x)=\pi(h)*\psi(x)$. Now let $\lambda \in p^{-1}(T_{0,0})$ and write $p(\lambda)=(t,0,t)$. By Condition~\ref{i(T times Gamma^0)=p^{-1}(Gamma^0)} of Definition~\ref{defn of top twist}, $\lambda=i(z,t)$.
As in the convolution formula following Definition \ref{defn of top twist}
where $(e,1,\sigma(e)) \mapsto \lambda_e$ is a section of $p$ over the image of $j$
we compute
\begin{align*}
\psi(x)^**\psi(y)(\lambda)&=   
\sum_{\sigma(e)=t}\overline{x(\lambda_e\lambda^{-1})}y(\lambda_e)
\\&=\sum_{\sigma(e)=t}\overline{x(\overline{z} \cdot \lambda_e)}y(\lambda_e)
\\&=z\langle x,y\rangle_{C_0(T)}(t)
\\&=\pi(\langle x,y\rangle_{C_0(T)})(\lambda).
\end{align*}
So $\psi(x)^**\psi(y)=\pi(\langle x,y\rangle_{C_0(T)})$.
Hence $\psi$ is bounded with the unique extension $\psi$ to $X(E,\mathbf{B}_E)$, the twisted graph correspondence over $C_0(T)$ obtained as the completion of the equivariant complex-valued continuous functions with compact support on $\mathbf{B}$;
moreover, $(\psi,\pi)$ is an injective representation of $X(E,\mathbf{B}_E)$ in $C^*(\Gamma(T,\sigma),\Lambda)$.

Now we prove that $(\psi,\pi)$ is covariant. By Definition~\ref{define the top graph}, we have $E_{\mathrm{rg}}^0=E_{\mathrm{fin}}^0 \cap \overline{\dom(\sigma)}^\circ$.
By \cite[Lemma~1.22]{Katsura:TAMS04}, $E_{\mathrm{rg}}^0=\dom(\sigma)$.
By \cite[Proposition~3.10]{HuiLiTwisted1},
\[
\phi^{-1}(\mathcal{K}(X(E,\mathbf{B}))) \cap (\ker\phi)^\perp=C_0(E_{\mathrm{rg}}^0)=C_0(\dom(\sigma)).
\]
Fix a nonnegative function $f \in C_c(\dom(\sigma))$ such that $\sigma\vert_{\supp(f)}$ is injective and there is a continuous local section $\varphi:\supp(f) \to \mathbf{B}$.
In order to prove that $(\psi,\pi)$ is covariant, it is enough to show that $\psi^{(1)}(\phi(f))=\pi(f)$. By \cite[Lemma~4.63(c)]{RaeburnWilliams:Moritaequivalenceand98}, there exists a unique continuous map
\[
\tau:\{(\lambda,\lambda') \in \Lambda \times \Lambda:p(\lambda)=p(\lambda')\} \to \mathbb{T}
\]
such that $\tau(\lambda,\lambda')\cdot\lambda=\lambda'$. Define a map $x:\mathbf{B} \to \mathbb{C}$ by
\[
x(\lambda):= \begin{cases} \tau(\varphi(p(\lambda)),\lambda)\sqrt{f(p(\lambda))}, &\text{ if $\lambda \in p^{-1}(\supp(f))$}
   \\ 0, &\text{ otherwise}. \end{cases}
\]
It is straightforward to check that $x$ is an equivariant continuous function with compact support on $\mathbf{B}$, and $\phi(f)=\Theta_{x,x}$. Fix $\lambda \in p^{-1}(T_{0,0})$ and write $p(\lambda)=(t,0,t) \in \supp(f)$. Then
\begin{align*}
\psi(x)*\psi(x)^*(\lambda)&=\psi(x)(\lambda\lambda')\overline{\psi(x)(\lambda')}, \text{ where $p(\lambda')=(t,1,\sigma(t))$}
\\&=Q\circ i^{-1}(\lambda)f(p(\lambda))
\\&=\pi(f)(\lambda).
\end{align*}
So $(\psi,\pi)$ is covariant.

The existence of a $\mathbb{T}$-action $\beta$ on $C^*(\Gamma(T,\sigma),\Lambda)$ such that
$\beta_z(\pi(f)) = \pi(f)$ and $\beta_z(\psi(x))$ $= z\psi(x)$
for all $z \in \mathbb{T}$, $f \in C_0(T)$ and $x \in L_B$
follows by arguing as in \cite[Proposition~II.5.1]{Renault:groupoidapproachto80}.
It is straightforward to show that the $C^*$-algebra generated by the images of $\psi$ and $\pi$ exhausts
$C^*(\Gamma(T,\sigma),\Lambda)$.
Therefore by the gauge-invariant uniqueness theorem (see \cite[Theorem~6.4]{Katsura:JFA04}), the twisted topological graph algebra $\mathcal{O}(E,\mathbf{B})$ is isomorphic to the twisted groupoid $C^*$-algebra $C^*(\Gamma(T,\sigma),\Lambda)$.
\end{proof}

\section{Twisted Groupoid Models for Twisted Topological Graph Algebras}

In this section, we prove our main theorem.

\begin{lemma}\label{one-sided shift map is a local homeo}
Let $E$ be a topological graph. Denote by $\sigma:\partial E \setminus E_{\mathrm{sg}}^0 \to \partial E$ the one-sided shift map. Then $\sigma$ is a partial local homeomorphism on $\partial E$ with $\dom(\sigma)=\partial E \setminus E_{\mathrm{sg}}^0$.
\end{lemma}
\begin{proof}
For $\mu \in \partial E \setminus E_{\mathrm{sg}}^0$, take an open $s$-section $U$ (see Page~5) containing $\mu_1$. Then we have $\sigma(Z(U))=Z(s(U))$. It is straightforward to check that the restriction of $\sigma$ to $Z(U)$ is a homeomorphism onto $Z(s(U))$ in the subspace topologies.
\end{proof}

By Lemma~\ref{one-sided shift map is a local homeo}, we can define a new topological graph.

\begin{defn}
Let $E$ be a topological graph. Define a topological graph as follows.
\[
\widehat{E}=(\widehat{E}^0,\widehat{E}^1,\widehat{r},\widehat{s}):=(\partial E,\partial E \setminus E_{\mathrm{sg}}^0,\iota,\sigma).
\]
\end{defn}

\begin{lemma}\label{proper maps boundary paths}
Let $E$ be a topological graph. Then the range map $r:\partial E \to E^0$ is a proper continuous surjection. Define a projection map $Q: \partial E \setminus E_{\mathrm{sg}}^0 \to E^1$ by $Q(\mu):=\mu_1$. Then $Q$ is also a proper continuous surjection.
\end{lemma}
\begin{proof}
First, we prove that $r$ is a proper continuous surjection. By Condition~(\ref{r(mu^n) to r(mu)}) of Lemma~\ref{convergent net in par E}, $r$ is continuous. By \cite[Lemma~1.4]{Katsura:ETDS06}, $r$ is surjective. For any compact subset $K \subset E^0$, we have $r^{-1}(K)$ is compact because $r^{-1}(K)=Z(K)$
(note $Z(K)$ is compact by \cite[Proposition 3.15]{Yeend:JOT07}). So $r$ is proper.

Now we prove that $Q$ is a proper continuous surjection. By Condition~(\ref{mu^n_i to mu_i}) of Lemma~\ref{convergent net in par E}, $Q$ is continuous. By \cite[Lemma~1.4]{Katsura:ETDS06}, $Q$ is surjective. For any compact subset $K \subset E^1$, we have $Q^{-1}(K)$ is compact because $Q^{-1}(K)=Z(K)$. So $Q$ is proper.
\end{proof}

Let $E$ be a topological graph and let $p:\mathbf{B} \to E^1$ be a principal circle bundle.
We get a principal circle bundle $Q^*(p):Q^*(\mathbf{B}) \to \partial E \setminus E_{\mathrm{sg}}^0$
which is the pullback bundle of $\mathbf{B}$ by $Q$.
Then there is a linear map $Q_*: X(E, B) \to X(\widehat{E}, Q^*(\mathbf{B}))$ obtained as the extension of the natural map $Q_*:C_c^e(\mathbf{B}) \to C_c^e(Q_*(\mathbf{B}))$ induced by $Q$ and a homomorphism $r_*^0:C_0(E^0) \to C_0(\partial E)$ induced from $r$.
Let $(j_X,j_A)$ be the universal covariant  representation of $X(E,\mathbf{B})$ in $\mathcal{O}(E,\mathbf{B})$, and let $(j_{X,\widehat{E}},j_{A,\widehat{E}})$ be the universal covariant  representation of $X(\widehat{E},Q^*(\mathbf{B}))$ in $\mathcal{O}(\widehat{E},Q^*(\mathbf{B}))$.

We next apply Proposition~\ref{homo bet twisted top graph alg} to obtain a homomorphism
$h:\mathcal{O}(E,\mathbf{B}) \to \mathcal{O}(\widehat{E},Q^*(\mathbf{B}))$.

\begin{lemma}\label{h is injective}
With notation as above the pair $(r,Q)$ defines a regular factor map from $\widehat{E}$ to $E$.
And the pair  $(j_{X,\widehat{E}} \circ Q_*,j_{A,\widehat{E}} \circ r_*)$ is a covariant  representation of
$X(E,\mathbf{B})$ in $\mathcal{O}(\widehat{E},Q^*(\mathbf{B}))$.  Hence there is a unique homomorphism
$h:\mathcal{O}(E,\mathbf{B}) \to \mathcal{O}(\widehat{E},Q^*(\mathbf{B}))$ such that
$h \circ j_X=j_{X,\widehat{E}} \circ Q_*$ and $h \circ j_A=j_{A,\widehat{E}} \circ r_*$.
Moreover, $h$ is injective.
\end{lemma}
\begin{proof}
By Lemma~\ref{proper maps boundary paths}  $(r,Q)$ defines a factor map from $\widehat{E}$ to $E$.
Note that
\[
\widehat{E}_{\mathrm{rg}}^0 = \dom\,\sigma = \widehat{E}^1 =\partial E \setminus E_{\mathrm{sg}}^0
\]
and so $\widehat{E}_{\mathrm{sg}}^0 = E_{\mathrm{sg}}^0$.
Hence, $r(\widehat{E}_{\mathrm{sg}}^0)  = E_{\mathrm{sg}}^0$ and so $(r,Q)$ is regular.
Therefore by Proposition~\ref{homo bet twisted top graph alg} the pair
$(j_{X,\widehat{E}} \circ Q_*,j_{A,\widehat{E}} \circ r_*)$ is a covariant representation of
$X(E,\mathbf{B})$ in $\mathcal{O}(\widehat{E},Q^*(\mathbf{B}))$ and there exists a unique map
$h:\mathcal{O}(E,\mathbf{B})$ $\to \mathcal{O}(\widehat{E},Q^*(\mathbf{B}))$ with the prescribed properties.
Since $r$ and $Q$ are both surjective, the injectivity of $h$ follows by the same result.
\end{proof}


The following theorem is inspired by \cite[Proposition~5.5]{Yeend:CM06}.

\begin{thm}\label{twisted top graph alg from SGDS}
The map $h:\mathcal{O}(E,\mathbf{B}) \to \mathcal{O}(\widehat{E},Q^*(\mathbf{B}))$
above is an isomorphism.
\end{thm}
\begin{proof}
Since $h$ is injective by  Lemma \ref{h is injective} we need only show that $h$ is surjective.
It is sufficient to prove that the image of $h$ contains the images of
$j_{X,\widehat{E}}$ and $j_{A,\widehat{E}}$.

Firstly we show that the image of $h$ contains the image of $j_{A,\widehat{E}}$. By the Stone-Weierstrass Theorem, we only need to prove that for each $\mu \in \partial E$ there exists $f \in C_0(\partial E)$ satisfying $f(\mu) \neq 0$ and $j_{A,\widehat{E}}(f) \in h(\mathcal{O}(E,\mathbf{B}))$, and that the image of $h$ separates points of $\partial E$.

Fix $\mu \in \partial E$. By the Urysohn's Lemma, there exists $f \in C_0(E^0)$ such that $f(r(\mu))=1$. Then $j_{A,\widehat{E}} \circ r_*(f)=h \circ j_{A}(f) \in  h(\mathcal{O}(E,\mathbf{B}))$, and $r_*(f)(\mu)=f(r(\mu)) =1$.

Now we prove that $h$ separates points of $\partial E$. Fix distinct $\mu, \nu \in \partial E$.

Case $1$. $r(\mu) \neq r(\nu)$. Take an arbitrary $f \in C_0(E^0)$ such that $f(r(\mu)) \neq f(r(\nu))$. Then $j_{A,\widehat{E}} \circ r_*(f)=h \circ j_{A}(f) \in  h(\mathcal{O}(E,\mathbf{B}))$, and $r_*(f)(\mu) \neq r_*(f)(\nu)$.

Case $2$. $\mu \in E_{\mathrm{sg}}^0, \nu \notin E_{\mathrm{sg}}^0$, and $r(\nu)=\mu$. Take a precompact open $s$-section $U$ of $\nu_1$ which admits a local section $\varphi:U \to \mathbf{B}$. Take an arbitrary $x \in C_c^e(p^{-1}(U))$ such that $x$ does not vanish on the fibre $p^{-1}(\nu_1)$. Define $f: Q^{-1}(U) \to \mathbb{C}$ by $f(\alpha):=\vert x \circ \varphi(\alpha_1)\vert^2$. Then $f \in C_c(Q^{-1}(U))$. So
\begin{align*}
h \circ j_X(x) (h \circ j_X(x))^*&=j_{X,\widehat{E}} \circ Q_*(x)(j_{X,\widehat{E}} \circ Q_*(x))^*
\\&=j_{X,\widehat{E}}^{(1)}(\Theta_{Q_*(x),Q_*(x)})
\\&=j_{X,\widehat{E}}^{(1)}(\phi(f))
\\&=j_{A,\widehat{E}}(f) \text{ (By the covariance of $(j_{X,\widehat{E}},j_{A,\widehat{E}})$). }
\end{align*}
Notice that $f(\mu)=0$ and $f(\nu) \neq 0$.

Case $3$. $r(\mu)=r(\nu), \mu, \nu \notin E_{\mathrm{sg}}^0, \mu_1 \neq \nu_1$. Take a precompact open $s$-section $U$ of $\nu_1$ which does not contains $\mu_1$ and admits a local section $\varphi:U \to \mathbf{B}$. Take an arbitrary $x \in C_c^e(p^{-1}(U))$ such that $x$ does not vanish on the fibre $p^{-1}(\nu_1)$. Define $f: Q^{-1}(U) \to \mathbb{C}$ by $f(\alpha):=\vert x \circ \varphi(\alpha_1)\vert^2$. Then $f \in C_c(Q^{-1}(U))$. Similar arguments from Case~2 gives $j_{A,\widehat{E}}(f) \in h(\mathcal{O}(E,\mathbf{B}))$. Notice that $f(\mu)=0$ and $f(\nu) \neq 0$.

Case $4$. $\vert\mu\vert=n \geq 1, \vert\nu\vert \geq n+1,$ and $\mu=\nu_1\cdots\nu_n$. For $1 \leq i \leq n+1$. Take a precompact open $s$-section $U_i$ of $\nu_i$ which admits a local section $\varphi_i:U_i \to \mathbf{B}$. Take an arbitrary $x_i \in C_c^e(p^{-1}(U_i))$ such that $x_i$ does not vanish on the fibre $p^{-1}(\nu_i)$. Define $f_i: Q^{-1}(U_i) \to \mathbb{C}$ by $f_i(\alpha):=\vert x_i \circ \varphi_i(\alpha_1)\vert^2$. Then $f_i \in C_c(Q^{-1}(U_i))$. So
\begin{align*}
\Big(\prod_{i=1}^{n+1}h \circ j_X(x_i) \Big)\Big(\prod_{i=1}^{n+1}h \circ j_X(x_i)\Big)^*&=\Big(\prod_{i=1}^{n+1}j_{X,\widehat{E}} \circ Q_*(x_i) \Big)\Big(\prod_{i=1}^{n+1}j_{X,\widehat{E}} \circ Q_*(x_i)\Big)^*
\\&=j_{A,\widehat{E}}(f_1 \cdots (f_n \circ \sigma^{n-1})( f_{n+1} \circ \sigma^n)).
\end{align*}
Notice that $f_1 \cdots (f_n \circ \sigma^{n-1})( f_{n+1} \circ \sigma^n)(\mu)=0$ and $f_1 \cdots (f_n \circ \sigma^{n-1})( f_{n+1} \circ \sigma^n)$ $(\nu) \neq 0$.

Case $5$. $\vert\mu\vert, \vert\nu\vert \geq n+1 (n \geq 1), \mu_1 \cdots\mu_n=\nu_1 \cdots\nu_n$, and $\mu_{n+1}\neq\nu_{n+1}$. For $1 \leq i \leq n$. Take a precompact open $s$-section $U_i$ of $\nu_i$ which admits a local section $\varphi_i:U_i \to \mathbf{B}$. Take an arbitrary $x_i \in C_c^e(p^{-1}(U_i))$ such that $x_i$ does not vanish on the fibre $p^{-1}(\nu_i)$. Define $f_i: Q^{-1}(U_i) \to \mathbb{C}$ by $f_i(\alpha):=\vert x_i \circ \varphi_i(\alpha_1)\vert^2$. Then $f_i \in C_c(Q^{-1}(U_i))$. Take a precompact open $s$-section $U_{n+1}$ of $\nu_{n+1}$ which does not contain $\mu_{n+1}$ and admits a local section $\varphi_{n+1}:U_{n+1} \to \mathbf{B}$. Take an arbitrary $x_{n+1} \in C_c^e(p^{-1}(U_{n+1}))$ such that $x_{n+1}$ does not vanish on the fibre $p^{-1}(\nu_{n+1})$. Define $f_{n+1}: Q^{-1}(U_{n+1}) \to \mathbb{C}$ by $f_{n+1}(\alpha):=\vert x_{n+1} \circ \varphi_{n+1}(\alpha_1)\vert^2$. Then $f_{n+1} \in C_c(Q^{-1}(U_{n+1}))$. Similar arguments from Case~4 implies that
\begin{align*}
\Big(\prod_{i=1}^{n+1}h \circ j_X(x_i) \Big)\Big(\prod_{i=1}^{n+1}h \circ j_X(x_i)\Big)^*=j_{A,\widehat{E}}(f_1 \cdots (f_n \circ \sigma^{n-1})( f_{n+1} \circ \sigma^n)).
\end{align*}
Notice that $f_1 \cdots (f_n \circ \sigma^{n-1})( f_{n+1} \circ \sigma^n)(\mu)=0$ and $f_1 \cdots (f_n \circ \sigma^{n-1})( f_{n+1} \circ \sigma^n)$ $(\nu) \neq 0$.

Therefore we deduce that the image of $h$ separates points of $\partial E$, and that the image of $h$ contains the image of $j_{A,\widehat{E}}$.

Now we show that the image of $h$ contains the images of $j_{X,\widehat{E}}$. Fix $x \in C_c^e(Q^*(\mathbf{B}))$. Take a finite cover $\{U_i\}_{i=1}^{n}$ of $(Q \circ Q^*(P))(\supp (x))$ by precompact open $s$-sections such that for each $i$ there exists a local section $\varphi_i:U_i \to \mathbf{B}$. Take a finite collection $\{h_i\}_{i=1}^{n} \subset C_c(E^1)$ such that $\supp(h_i) \subset U_i, \sum_{i=1}^{n}h_i=1$ on $(Q \circ Q^*(P))(\supp (x))$. Since each $((Q \circ Q^*(P))_*(h_i))x \in  C_c^e(Q^*(\mathbf{B}))$ and $\sum_{i=1}^{n}((Q \circ Q^*(P))_*(h_i))x=x$, we may assume that $(Q \circ Q^*(P))(\supp (x))$ is contained in a precompact open $s$-section $U$ which admits a local section $\varphi:U \to \mathbf{B}$.

Take an arbitrary $y \in C_c^e(p^{-1}(U))$ such that $y(b)=b/\varphi(p(b))$ for all $b \in p^{-1}((Q \circ Q^*(P))(\supp(x)))$. Define $f:r^{-1}(s(U)) \to \mathbb{C}$ by $f(\mu):=x(\varphi \circ s\vert_U^{-1} \circ r(\mu),(s\vert_U^{-1} \circ r(\mu))\mu)$. Then $f \in C_c(r^{-1}(s(U)))$. We claim that $Q_*(y) \cdot f=x$. Fix $(b,e\nu) \in Q^*(\mathbf{B})$.

Case $1$. $(b,e\nu) \notin \supp(x)$. Then $x(b,e\nu)=0$. If $b \notin p^{-1}(U)$, then $Q_*(y)$ $(b,e\nu)=0$, so $(Q_*(y)\cdot f)(b,e\nu)=0$. If $b \in p^{-1}(U)$, then $\nu \in r^{-1}(s(U))$, so
\[
f(\nu)=x(\varphi \circ s\vert_U^{-1} \circ r(\nu),(s\vert_U^{-1} \circ r(\nu))\nu)=x(\varphi(e),e\nu)=(\varphi(e)/b) x(b,e\nu)=0.
\]

Case $2$. $(b,e\nu) \in \supp(x)$. We compute that
\[
(Q_*(y) \cdot f)(b,e\nu)=y(b)f(\nu)=(b/\varphi(e))(\varphi(e)/b) x(b,e\nu)=x(b,e\nu).
\]
So $Q_*(y) \cdot f=x$ and we finish proving the claim. Hence
\[
h(j_X(y))j_{A,\widehat{E}}(f)=j_{X,\widehat{E}}(Q_*(y))j_{A,\widehat{E}}(f)=j_{X,\widehat{E}}(x).
\]
Therefore the image of $h$ contains the image of $j_{X,\widehat{E}}$ because we just showed that the image of $h$ contains the image of $j_{A,\widehat{E}}$. We are done.
\end{proof}

Recall that by Lemma \ref{one-sided shift map is a local homeo}, the shift map $\sigma$ is
a partial  local homeomorphism on $\partial E$.
\begin{defn}\label{boundary-path gpoid}
The \emph{boundary path groupoid} of a topological graph $E$ is defined to be the Renault-Deaconu groupoid
$\Gamma(\partial E,\sigma)$ (see Definition~\ref{define D-R groupoid}).
\end{defn}

\begin{thm}\label{twisted top graph alg iso to twisted gpoid C*-alg}
Let $E$ be a topological graph and let $p:\mathbf{B} \to E^1$ be a principal circle bundle.
Let $Q^*(p):Q^*(\mathbf{B}) \to \partial E \setminus E_{\mathrm{sg}}^0$ be the pullback bundle of $\mathbf{B}$ by $Q$.
Denote by $j:\partial E \setminus E_{\mathrm{sg}}^0 \to \Gamma(\partial E,\sigma)$ the embedding such that $j(e \nu)=(e \nu,1,\nu)$ for all $e \in E^1, \nu \in \partial E$ with $s(e)=r(\nu)$.
Let $\Lambda$ be the topological twist $\Lambda$ over the boundary path groupoid $\Gamma(\partial E,\sigma)$
\[
\mathbb{T} \times \Gamma^0 \xrightarrow{i} \Lambda \xrightarrow{p'} \Gamma(\partial E,\sigma)
\]
such that $j^*(\Lambda) \cong Q^*(\mathbf{B})$
(see Theorem~\ref{1-1 corr bet twisted groupoid and line bundle}).
Then $\mathcal{O}(E,\mathbf{B})$ is isomorphic to the twisted groupoid $C^*$-algebra $C^*(\Gamma(\partial E,\sigma),\Lambda)$.
\end{thm}
\begin{proof}
The result follows directly from Theorems~\ref{The twisted top graph alg is iso to the twisted gpoid C*-alg}, \ref{twisted top graph alg from SGDS}.
\end{proof}

\begin{eg}
In 1989 Rieffel introduced quantum Heisenberg manifolds $D_{\mu, \nu}^c$, where $\mu, \nu \in \RR$ and $c \in \NN$
as key examples of his deformation quantization theory (see \cite{Rieffel}).
Work of Abadie et al.\ (see \cite{Abadie-Exel-Eilers}) showed that each quantum Heisenberg manifolds
$D_{\mu, \nu}^c$ is isomorphic to a twisted topological graph  $C^*$-algebra $\Oo_{X(E, L)}$
(without using the language of topological graphs) with $E^0 = E^1 = \TT^2$, $r = \text{id}$, $s$
is translation by a parameter depending on $\mu, \nu \in \RR$
and $L$ is a Hermitian line bundle determined by the integer $c$.
Kang et al.\  (see \cite{MR3303906}) proved that $D_{\mu, \nu}^c$ is a twisted groupoid $C^*$-algebra.
\end{eg}

\section*{Appendix}

In this appendix, we provide an alternative proof of Theorem~\ref{1-1 corr bet twisted groupoid and line bundle} by using the cocycles approach.

Firstly, we can present the principal circle bundle in the following way. There exist an open cover $\{N_\alpha\}_{\alpha \in \Theta}$ of $\dom(\sigma)$ and a $1$-cocycle $\{s_{\alpha\beta}\}_{\alpha,\beta \in \Theta}$, such that
\[
\mathbf{B} \cong \amalg_{\alpha \in \Theta}(N_\alpha \times \mathbb{T})/(t,z,\alpha)\sim(t,zs_{\alpha\beta}(t),\beta).
\]
 For $k_1,k_2 \geq 1$, we have a principal circle bundle over $(\prod_{i=1}^{k_1}\dom(\sigma)) \times (\prod_{j=1}^{k_2}$ $\dom(\sigma))$
\begin{align*}
\amalg \Big( \Big(\prod_{i=1}^{k_1}N_{\alpha_i}\Big) \times \Big(\prod_{j=1}^{k_2}N_{\alpha_j '}\Big) \times &\mathbb{T}\Big) \Big/ (t_1,\dots,t_{k_1},t_{1}',\dots,t_{k_2}',z,\alpha_1,\dots,\alpha_{k_1},\alpha_{1}',\dots,\alpha_{k_2}') \sim \\&(t_1,\dots,t_{k_1},t_{1}',\dots,t_{k_2}',zs_{\alpha_1\beta_1}(t_1)\cdots s_{\alpha_{k_1}\beta_{k_1}}(t_{k_1})
\\&s_{\beta_{1}'\alpha_{1}'}(t_{1}')\cdots s_{\beta_{k_2}'\alpha_{k_2}'}(t_{k_2}'),\beta_1,\dots,\beta_{k_1},\beta_{1}',\dots,\beta_{k_2}').
\end{align*}
Notice that there is an embedding $\iota_{k_1,k_2}:\Gamma_{k_1,k_2} \to (\prod_{i=1}^{k_1}\dom(\sigma)) \times (\prod_{j=1}^{k_2}\dom$ $(\sigma))$ by sending $(t_1,k_1-k_2,t_2)$ to $(t_1,\dots,\sigma^{k_1-1}(t_1),t_2,\dots,\sigma^{k_2-1}(t_2))$ for all $t_1 \in \dom(\sigma^{k_1}), t_2 \in \dom(\sigma^{k_2})$. Define a principal circle bundle $p_{k_1,k_2}:\Lambda_{k_1,k_2} \to\Gamma_{k_1,k_2}$ to be the restriction of the above bundle to $\Gamma_{k_1,k_2}$, that is
\[
\Lambda_{k_1,k_2}:=\{(t_1,\dots,\sigma^{k_1-1}(t_1),t_2,\dots,\sigma^{k_2-1}(t_2),z,\alpha_1,\dots,\alpha_{k_1},\alpha_{1}',\dots,\alpha_{k_2}')\}.
\]

For $k \geq 1$, there are embeddings $\iota_{k,0}:\Gamma_{k,0} \to \prod_{i=1}^{k}\dom(\sigma); \iota_{0,k}:\Gamma_{0,k} \to \prod_{i=1}^{k}\dom(\sigma)$, and similarly we get principal circle bundles $\Lambda_{k,0}$ over $\Gamma_{k,0}$; $\Lambda_{0,k}$ over $\Gamma_{0,k}$.

Moreover, we regard $\Gamma_{0,0}$ as a copy of $T$ via the homeomorphism $\iota_{0,0}:\Gamma_{0,0} \to T$. Denote by $\Lambda_{0,0}$ the trivial principal circle bundle $T \times \mathbb{T}$ over $T$.

For $k_1, k_2 \geq 1$, define $h_{(k_1,k_2),(k_1,k_2)}:=\id$.

For $1 \leq k_1< l_1,1 \leq k_2 < l_2$ with $k_1-k_2=l_1-l_2$, define
\[
h_{(k_1,k_2),(l_1,l_2)}:p_{k_1,k_2}^{-1}(\Gamma_{k_1,k_2} \cap \Gamma_{l_1,l_2}) \to p_{l_1,l_2}^{-1}(\Gamma_{k_1,k_2} \cap \Gamma_{l_1,l_2})
\]
as follows. For any $(t_1,\dots,\sigma^{k_1-1}(t_1),t_2,\dots,\sigma^{k_2-1}(t_2),z,\alpha_1,\dots,\alpha_{k_1},\alpha_1',\dots, \alpha_{k_2}') \in p_{k_1,k_2}^{-1}$ $(\Gamma_{k_1,k_2}$ $\cap \Gamma_{l_1,l_2})$, choose arbitrary $\alpha_{k_1+1},\dots,\alpha_{l_1},\alpha_{k_2+1}',\dots,\alpha_{l_2}'$ such that $\sigma^{k_1-1+i}$ $(t_1) \in N_{\alpha_{k_1+i}}\cap N_{\alpha_{k_2+i}'}, i=1,\dots,l_1-k_1$. Define
\begin{align*}
&h_{(k_1,k_2),(l_1,l_2)}(t_1,\dots,\sigma^{k_1-1}(t_1),t_2,\dots,\sigma^{k_2-1}(t_2),z,\alpha_1,\dots,\alpha_{k_1},\alpha_1',\dots,\alpha_{k_2}'):=
\\&(t_1,\dots,\sigma^{l_1-1}(t_1),t_2,\dots,\sigma^{l_2-1}(t_2),zs_{\alpha_{k_2+1}'\alpha_{k_1+1}}(\sigma^{k_1}(t_1))\cdots s_{\alpha_{l_2}'\alpha_{l_1}}(\sigma^{l_1-1}(t_1)),
\\&\alpha_1,\dots,\alpha_{l_1},\alpha_1',\dots,\alpha_{l_2}').
\end{align*}
It is straightforward to prove that $h_{(k_1,k_2),(l_1,l_2)}$ is a homeomorphism. Denote its inverse by $h_{(l_1,l_2),(k_1,k_2)}$ with the formula given as follows. For $(t_1,\dots,\sigma^{l_1-1}(t_1),t_2$, $\dots,\sigma^{l_2-1}(t_2),z, \alpha_1$, $\dots, \alpha_{l_1},\alpha_1',\dots,\alpha_{l_2}') \in p_{l_1,l_2}^{-1}(\Gamma_{k_1,k_2} \cap \Gamma_{l_1,l_2})$,
\begin{align*}
&h_{(l_1,l_2),(k_1,k_2)}(t_1,\dots,\sigma^{l_1-1}(t_1),t_2,\dots,\sigma^{l_2-1}(t_2),z,\alpha_1,\dots,\alpha_{l_1},\alpha_1',\dots,\alpha_{l_2}')
\\&=(t_1,\dots,\sigma^{k_1-1}(t_1),t_2,\dots,\sigma^{k_2-1}(t_2),zs_{\alpha_{k_1+1}\alpha_{k_2+1}'}(\sigma^{k_1}(t_1))\cdots s_{\alpha_{l_1}\alpha_{l_2}'}(\sigma^{l_1-1}(t_1)),\\&\alpha_1,\dots,\alpha_{k_1},\alpha_1',\dots,\alpha_{k_2}').
\end{align*}

Similarly, for any $k_1,k_2,l_1,l_2 \geq 0$ with $k_1-k_2=l_1-l_2$, we are able to define a homeomorphism $h_{(k_1,k_2),(l_1,l_2)}$.

It is straightforward to check that for $k_1,k_2,l_1,l_2,m_1,m_2 \geq 0$ with $k_1-k_2=l_1-l_2=m_1-m_2$, we have $p_{l_1,l_2} \circ h_{(k_1,k_2),(l_1,l_2)}=p_{k_1,k_2}$, and $h_{(l_1,l_2),(m_1,m_2)} \circ h_{(k_1,k_2),(l_1,l_2)}=h_{(k_1,k_2),(m_1,m_2)}$ on $p_{k_1,k_2}^{-1}(\Gamma_{k_1,k_2} \cap \Gamma_{l_1,l_2} \cap \Gamma_{m_1,m_2})$.
By Lemma~\ref{piece principal bundles open version}, we may construct a locally compact Hausdorff space $\Lambda_z$ for
$z \in \mathbb{Z}$ by
\[
\Lambda_z:=\amalg_{\{k_1,k_2 \geq 0: k_1-k_2=z\}}\Lambda_{k_1,k_2}/\{\lambda \sim h_{(k_1,k_2),(l_1,l_2)}(\lambda):\lambda \in p_{k_1,k_2}^{-1}(\Gamma_{k_1,k_2} \cap \Gamma_{l_1,l_2})\}.
\]
For $k_1,k_2,l_1,l_2 \geq 0$, if $k_1-k_2 \neq l_1-l_2$, then $\Gamma_{k_1,k_2} \cap \Gamma_{l_1,l_2}=\emptyset$. So we get a locally compact Hausdorff space $\Lambda:=\amalg_{z \in \mathbb{Z}}\Lambda_z$.

Now we endow $\Lambda$ with a groupoid structure. For $k_i \geq 1, t_i \in \dom(\sigma^{k_i}),i=1,2,3$, for $z_1,z_2 \in \mathbb{T}$, suppose that $(t_1,\dots,\sigma^{k_1-1}(t_1)) \in \prod_{i=1}^{k_1}N_{\alpha_i},(t_2,\dots,\sigma^{k_2-1}$ $(t_2)) \in \prod_{i=1}^{k_2}(N_{\alpha_i'} \cap N_{\alpha_i''})$, and that $(t_3,\dots,\sigma^{k_3-1}(t_3)) \in \prod_{i=1}^{k_3}N_{\alpha_i'''}$, define
\begin{align*}
(t_1,&\dots,\sigma^{k_1-1}(t_1),t_2,\dots,\sigma^{k_2-1}(t_2),z_1,\alpha_1,\dots,\alpha_{k_1},\alpha_{1}',\dots,\alpha_{k_2}') \cdot
\\&(t_2,\dots,\sigma^{k_2-1}(t_2),t_3,\dots,\sigma^{k_3-1}(t_3),z_2,\alpha_1'',\dots,\alpha_{k_2}'',\alpha_{1}''',\dots,\alpha_{k_3}''')
\\&:=(t_1,\dots,\sigma^{k_1-1}(t_1),t_3,\dots,\sigma^{k_3-1}(t_3),z_1z_2s_{\alpha_1'' \alpha_1'}(t_2)\cdots s_{\alpha_{k_2}'' \alpha_{k_2}'}(\sigma^{k_2-1}(t_2)),
\\&\alpha_1,\dots,\alpha_{k_1},\alpha_{1}''',\dots,\alpha_{k_3}''');
\end{align*}
define
\begin{align*}
(t_1,\dots,&\sigma^{k_1-1}(t_1),t_2,\dots,\sigma^{k_2-1}(t_2),z_1,\alpha_1,\dots,\alpha_{k_1},\alpha_{1}',\dots,\alpha_{k_2}')^{-1}
\\&:=(t_2,\dots,\sigma^{k_2-1}(t_2),t_1,\dots,\sigma^{k_1-1}(t_1),\overline{z_1},\alpha_1',\dots,\alpha_{k_2}',\alpha_{k_1},\dots,\alpha_1).
\end{align*}
More simply,
\begin{align*}
(t_1,&\dots,\sigma^{k_1-1}(t_1),t_2,\dots,\sigma^{k_2-1}(t_2),z_1,\alpha_1,\dots,\alpha_{k_1},\alpha_{1}',\dots,\alpha_{k_2}') \cdot
\\&(t_2,\dots,\sigma^{k_2-1}(t_2),t_3,\dots,\sigma^{k_3-1}(t_3),z_2,\alpha_1',\dots,\alpha_{k_2}',\alpha_{1}''',\dots,\alpha_{k_3}''')
\\&:=(t_1,\dots,\sigma^{k_1-1}(t_1),t_3,\dots,\sigma^{k_3-1}(t_3),z_1z_2,\alpha_1,\dots,\alpha_{k_1},\alpha_{1}''',\dots,\alpha_{k_3}''').
\end{align*}
It is straightforward to check that $\Lambda$ is a locally compact groupoid under these two operations with the unit space $\Lambda^0$ which is homeomorphic to $\Gamma^0$. Define $i:\Gamma^0 \times \mathbb{T} \to \Lambda$ to be the embedding such that its image is $\Lambda_{0,0}$. Define $p':\Lambda \to \Gamma(T,\sigma)$ in the obvious way. Thus $\Lambda$ is the desired topological twist in Theorem~\ref{1-1 corr bet twisted groupoid and line bundle}.

In \cite{MR3303906} Kang et al.\  constructed $\Lambda$ by using cocycles for the case when $\sigma$ is a homemorphism and $T$ is a compact metric space.

\section*{Acknowledgments}

The second author would like to thank his PhD supervisors Professor David Pask and Professor Aidan Sims for supporting his trip to US to start this research with the first author. The second author in particular wants to thank the first author for lots of encouragements and for many helpful conversations. The second author also appreciates the hospitality of Department of Mathematics and Statistics, University of Nevada, Reno during his visit. The first author would like to thank the second author for all his hard work and for inviting him to work on this project. The authors appreciate the hospitality of Department of Mathematics, University of Wyoming during their visit.

\end{document}